\documentclass[aap]{imsart}

\RequirePackage{amsthm,amsmath,amsfonts,amssymb}
\RequirePackage[numbers,sort&compress]{natbib}
\RequirePackage[colorlinks,citecolor=blue,urlcolor=blue]{hyperref}
\RequirePackage{graphicx}

\usepackage{fix-cm}

\usepackage{latexsym,upgreek}
\usepackage{amsfonts}
\usepackage{mathrsfs}
\usepackage{amsmath}
\usepackage{amsthm}
\usepackage{graphicx} 
\usepackage{amssymb}
\usepackage{stmaryrd}
\usepackage{natbib}
\usepackage{bbm}
\usepackage{commath}
\usepackage{lipsum}
\usepackage{float} 
\usepackage{enumitem}
\usepackage{comment} 
\usepackage{cancel}
 
\startlocaldefs

\newcommand{\rw}{\rightarrow}    
\newcommand{\Real}{\mbR}

\newcommand{\mB}{{\mathcal B}}
\newcommand{\mP}{{\mathcal P}}
\newcommand{\mK}{{\mathcal K}}

\newcommand{\mF}{{\mathcal F}}
\newcommand{\mE}{{\mathcal E}}
\newcommand{\mX}{{\mathcal X}}
\newcommand{\mY}{{\mathcal Y}}
\newcommand{\mN}{{\mathcal N}}
\newcommand{\mM}{{\mathcal M}}

\newcommand{\mO}{{\mathcal O}}

\newcommand{\mbN}{\mathbb{N}}

\newcommand{\mbP}{\mathbb{P}}
\newcommand{\mbR}{\mathbb{R}}
 
\newcommand{\mbE}{\mathbb{E}} 

\newcommand{\sd}{{\sf d}}

\newcommand{\sfd}{{\sf d}}

\newcommand{\beq}{\begin{equation}}
\newcommand{\eeq}{\end{equation}}
\newcommand{\beqa}{\begin{eqnarray}}
\newcommand{\eeqa}{\end{eqnarray}}
\newcommand{\nn}{\nonumber}

\newcommand{\cred}{\textcolor{red}}

\theoremstyle{plain}

\newtheorem{theorem}{Theorem}[section]
\newtheorem{lemma}[theorem]{Lemma}

\theoremstyle{definition}

\newtheorem{proposition}[theorem]{Proposition}
\newtheorem{corollary}[theorem]{Corollary}

\newtheorem{remark}{Remark}
\newtheorem{algorithm}{Algorithm}

\endlocaldefs

\begin{document}

\begin{frontmatter}
\title{
Importance sampling for Bayesian inference: polynomial-dimension dependent error bounds}
\runtitle{Importance sampling for Bayesian inference}

\begin{aug}
\author[A]{\fnms{Fabián}~\snm{González} \thanks{[\textbf{Corresponding author.}]}\ead[label=e1]{omgonzal@math.uc3m.es}\orcid{0000-0003-4353-5737}},
\author[B]{\fnms{Víctor}~\snm{Elvira}\ead[label=e2]{victor.elvira@ed.ac.uk}\orcid{0000-0002-8967-4866}}
\and
\author[A]{\fnms{Joaquín}~\snm{Míguez}\ead[label=e3]{joaquin.miguez@uc3m.es}\orcid{0000-0001-5227-7253}}
\address[A]{Department of Signal Theory and Communications, Universidad Carlos III de Madrid, Spain \printead[presep={ ,\ }]{e1,e3}}
\address[B]{School of Mathematics, University of Edinburgh, United Kingdom \printead[presep={,\ }]{e2}}
\end{aug}

\begin{abstract} Many Bayesian inference problems involve high-dimensional models where the performance of standard importance sampling (IS) methods often degrades rapidly as the dimensionality increases. Classical analyzes of IS typically rely on the assumption that observations are arbitrary but fixed (i.e., deterministic), thereby neglecting the probabilistic structure that the Bayesian model induces on the data. In this paper, we adopt the perspective that observations are themselves random variables whose distribution is governed by the underlying model. Within this probabilistic framework, we identify a model-dependent function, referred to as the \textit{link function}, which connects the
fixed- and random-observation formulations. 
We provide a characterization of the $L^2$ Monte Carlo estimation error: specifically we show that the $L^2$ error bounds are finite and converge at the standard Monte Carlo rate $\mO(N^{-1/2})$, for arbitrarily large dimension, if and only if the link function is Bochner integrable.
 This result reveals the fundamental quantity controlling the error of the approximation and establishes a mechanism to manage the dependence on the model state dimension. Consequently, our approach provides a principled way to alleviate the challenges of high dimensionality, offering insights that transcend the worst-case analyzes dominant in the existing literature. Finally, we derive explicit analytical examples of the dimensional scaling of the associated errors for several model classes, including linear–Gaussian systems and models with bounded observation functions.
\end{abstract}



\end{frontmatter}


\section{Introduction}\label{sIntroduction}

Bayesian inference is a cornerstone of modern statistical methodology, widely used in fields such as signal theory, image processing, and machine learning. Many problems in these domains are naturally framed in terms of high-dimensional probabilistic models, where solutions are given as integrals. Unfortunately, this integration can only be performed analytically for a limited set of models for which the prior distribution and the likelihood function are specifically selected to yield a closed-form posterior distribution, as it is the case for a handful of conjugate families \cite{Bernardo94}. The obvious alternative is numerical integration \cite{Robert04}.


\subsection{Motivation and background}
 
One widely used approach for handling intractable integrals in Bayesian inference is importance sampling (IS) \cite{Robert04,tokdar2010importance,Bugallo17}. IS is a Monte Carlo integration technique that employs a collection of random samples $X=x^i$, $i=1, \ldots, N$, drawn from a proposal probability measure $\nu$ to approximate integrals with respect to (w.r.t.) a target probability measure $\pi$. Each sample $x^i$ is assigned a weight $w^i$, which is 
\begin{itemize}
\item proportional to the relative density $\frac{\sfd \pi}{\sfd \nu}(x^i)$,
\item and normalised, $\sum_{i=1}^N w^i = 1$.
\end{itemize}
Under mild assumptions, weighted averages can be proved to converge to the corresponding integrals \cite{Robert04,Chopin20}, i.e., 
$$
\sum_{i=1}^N w^i \varphi(x^i) \stackrel{N\rw\infty}{\longrightarrow} \int \varphi(x)\pi(\sfd x),
$$ 
where $\varphi(\cdot)$ is a test function. 

IS is a classic methodology with a plethora of applications \cite{Doucet01b,Robert04,Bucklew04,DelMoral06,Bugallo17}. It provides a flexible and easy-to-use method for approximating expectations w.r.t. complex probability distributions.
Over the past decade, several advanced IS-based methods, \cite{Andrieu10,Chopin12,Crisan18bernoulli,Tran21,doucet2022score,doucet2022annealed,doucet2018sequential,akyildiz2021convergence} have been developed to extend IS and handle inference in increasingly sophisticated models, either static or dynamic. See also \cite{agapiou2017importance,chatterjee2018sample} for a recent account of the fundamentals of IS. 

One problem where IS schemes are often applied is Bayesian inference. In this setting, the target probability measure $\pi$ is the posterior distribution of a random variable (r.v.) $X$ (sometimes termed the state variable, or simply the state) conditional on some r.v. $Y$ that represents the observations available in the system of interest. To be specific, the objective is to determine the conditional probability law of $X$ given a realization $Y=y$. This is formulated via Bayes' theorem:
$$
\pi(\sd x) = \frac{l_y(x)\pi_0(\sd x)}{\int l_y(x')\pi_0(\sd x')} = \frac{l_y(x)\pi_0(\sd x)}{\pi_0(l_y)},
$$
where, $l_y(x)$ denotes the likelihood of the observation $Y=y$ given $X=x$, $\pi_0(\sd x)$ is the prior probability measure of the state $X$, and $\pi_0(l_y)=\int l_y(x')\pi_0(\sd x')$ is the marginal likelihood, or model evidence, which acts as the normalization constant for the posterior distribution $\pi(\sd x)$.

\subsection{Impact of dimension}

Despite its versatility, IS is known to suffer from limitations, most notably the so-called \textit{curse of dimensionality} \cite{agapiou2017importance,chatterjee2018sample,Rebeschini15,Snyder08}. Theoretical and empirical studies suggest that performance can deteriorate significantly when inference involves high-dimensional state spaces  (where the number of scalar r.v.'s or parameters to be inferred is referred to as the dimension of the state space).  
For example, in~\cite{chatterjee2018sample} the authors show that, in general, the number of samples required to approximate the posterior probability law $\pi(\sd x)$ (within a specified accuracy) can scale exponentially with the Kullback--Leibler (KL) divergence between the target and proposal distributions. Also, importance samplers can suffer from weight degeneracy when the dimension of the state space is large \cite{Bengtsson08,Snyder08,Snyder15}. While sophisticated methods have been proposed to overcome this difficulty in various setups \cite{VanLeeuwen13,Beskos14,Koblents15,Rebeschini15,Beskos17,Ruzayqat22}, they are often computationally costly and rely on assumptions which may be hard to assess.

Recent works have begun to characterise the multiple roles of dimension in the analysis of IS algorithms. In particular, the authors of ~\cite{agapiou2017importance} have identified a key quantity governing the efficiency of IS estimators: the second moment of the Radon--Nikodym derivative of the posterior (target) w.r.t. the prior (proposal) measure:
\[\frac{\pi_0(l_y^2)}{\pi_0(l_y)^2}\]
This quantity is further related to important metrics such as the effective sample size (ESS) \cite{Martino17,elvira2022rethinking}, and it also appears in bounds on the mean squared error of IS approximations \cite{akyildiz2021convergence}. It plays a key role in our analysis throughout this paper.

Other works have demonstrated that the degradation of performance in high-dimensional settings can be alleviated under certain conditions. Adaptive and structure-aware variants of IS (such as those found in the adaptive particle filtering literature~\cite{VanLeeuwen13,Rebeschini15,Beskos17,Ruzayqat22,kuntz2024divide} and in sequential Monte Carlo samplers~\cite{Beskos14}) exploit problem-specific geometric features or sequential data structure to maintain accuracy in moderately high dimensional problems.


\subsection{Contributions}

In this paper, we investigate the role of the dimension of the state $X$ (denoted $d_x$) in the
error bounds achieved by IS methods for Bayesian inference. Unlike much of the previous literature, which assumes fixed observations, we treat the data as r.v.s generated by the model. This approach has been employed in \cite{mcdonald2020exponential} for the study of stability of the optimal filter \cite{crisan2020stable}. However, to the best of our knowledge, it has not been applied to IS schemes or Monte Carlo approximations.
This natural probabilistic setting allows us to define random posterior measures and, correspondingly, random error bounds, enabling the use of functional-analytic and measure-theoretic tools to exploit the model structure.
We first derive $L^1$ and $L^2$ error bounds for general Bayesian models under random observations and explicitly identify their dependence on the dimension of the model. A central theoretical contribution of the paper is a characterization of
the $L^2$ Monte Carlo error of the IS estimator in this setting. We show that the canonical
$\mathcal{O}(N^{-1/2})$ convergence rate holds if and only if a certain
model-dependent function (refereed as the \textit{link function}) introduced in the analysis is square-integrable
w.r.t. the joint law of the observations and the prior
(formalized using Bochner integrability). This integrability
condition emerges as the fundamental quantity governing the error of
the approximation. In particular, it provides a direct mechanism to
control the dependence of the error bounds on the state dimension $d_x$. 
Our analysis relies on the notion of Bochner integrability from functional analysis \cite{tuomas2016analysis}, which, to the best of our knowledge, has not been previously used for the study of IS estimators.
 
We then focus on the dimensional scaling of these errors. For several classes of models, including linear/Gaussian systems, models with pointwise-bounded likelihoods and models with bounded observation functions. 
These examples are not artificial or contrived cases, but practically relevant models that arise naturally in a wide range of problems. For instance, sensors that measure physical magnitudes (pressure, velocity, voltage, etc.) typically have a limited range of operation, which translates into bounded observation functions when modelled mathematically. Consequently, these examples are representative of a broad class of settings to which our theoretical findings apply.

To derive these results, we first bridge the gap between abstract functional-analytic conditions and practical Bayesian modeling by establishing a concrete, easily verifiable criterion for the integrability of the link function. Specifically, for likelihoods belonging to elliptically symmetric  families \cite{delmas2024elliptically} (which encompass diverse noise models such as Gaussian, Laplace, Student-$t$ and Cauchy distributions) we prove that the high-dimensional Bochner integrability condition reduces to the finiteness of a simple, one-dimensional radial integral for models with bounded observations functions.

This reduction allows us to provide explicit analytical constructions showing that approximation errors grow at most polynomially with the state dimension $d_x$. Notably, we identify cases where the polynomial degree is zero, implying error bounds that hold uniformly over $d_x$. These findings demonstrate that the exponential complexity typically associated with the \textit{curse of dimensionality} can be mitigated in some practically relevant settings, such as those involving sensors with limited operational ranges.

This framework offers a straightforward diagnostic for practitioners: if the observation map is bounded and the radial profile of the noise satisfies a manageable tail-growth condition, polynomial or even uniform error scaling is guaranteed. By reducing the complex stability analysis of IS to the evaluation of a univariate integral, we provide a systematic method to verify the robustness of inference schemes across the various noise structures specified in Remark \ref{rem:Colour_noises}.


\subsection{Outline of the Paper}
The paper is organized as follows. Section~\ref{sec:sIS2} formally describes the probabilistic model and the classical importance sampling algorithm. In Section~\ref{sec:General_Random}, we analyze the error bounds attained for general models with random observations. The general results obtained in that framework are then particularised to some practically relevant examples in Section~\ref{sec:Special_cases}. Section \ref{sec:Comparison}, is devoted to a comparison of the results obtained in Sections \ref{sec:General_Random} and \ref{sec:Special_cases} with recent worst case analyzes of IS.
Finally, in Section~\ref{sec:Discussion}, we summarise our main findings and discuss possible venues for future research. Technical proofs and some additional results are provided in the appendices. We conclude this introduction with a brief summary of the notation used throughout the paper.


\subsection{Summary of notation} \label{ssNotation}
\begin{itemize}
    \item Sets, measures and integrals:
    \begin{itemize}
        \item[-] $\mB (S)$ is the $\sigma$-algebra of Borel subsets of $S \subseteq  \mbR^{d}$.
\item[-] $\mP(S) := \{ \nu  : \mB (S) \mapsto  [0, 1]$ and $\nu (S) = 1\}$  is the set of probability measures over
$\mB (S)$.
\item[-] $ \nu (f) := 
\int f(s) \nu(\sd s)$  is the integral of a Borel measurable function $f : S \mapsto  \mbR$ w.r.t. the
measure $\nu  \in  \mP (S)$.

\end{itemize}
\item Functions:
    \begin{itemize}
    
    \item [-] Consider the measurable spaces \((S, \mathcal{B}(S))\) and \((\mbR, \mathcal{B}(\mbR))\).  
We denote by \(B(S)\) the space of bounded, real-valued, measurable functions \(f : S \mapsto \mbR\).  
For any \(f \in B(S)\), the uniform norm is defined as
\[
\|f\|_{\infty} := \sup_{s \in S} |f(s)| < \infty.
\]
    \end{itemize}
    
\item Real r.v.s on a probability space $(\Omega,\mF,\mbP)$ are denoted by capital letters (e.g., $Z:\Omega \mapsto \mbR^{d})$, while their realisations are written as lowercase letters (e.g., $Z(\omega)=z$, or simply, $Z=z$). If $X$ is a $d_x$ multivariate Gaussian r.v., then its probability law is denoted $\mN(\sd x;\mu, \Sigma)$, and it associated pdf by $\mN( x;\mu, \Sigma)$ where $\mu$ is the mean and $\Sigma$ is the covariance matrix. If the r.v. $X$ has probability law $\pi$ and $f$ is a $\pi$-integrable function, then we denote $$\pi(f) = \int f(x) \pi(\sd x) = \mbE[f(X)],$$
where $\mbE[\cdot]$ is the expectation operator.

\item Linear algebra:
    \begin{itemize}
        \item Let $A \in \mbR^{n \times n}$ be a real symmetric matrix. The spectrum of $A$ is denoted as 
\[
\mathrm{spec}(A) := \{ \lambda \in \mbR \, : \, \exists\, v \in \mbR^n \setminus \{0\} \text{ such that } Av = \lambda v \}.
\]
\item We denote the real eigenvalues of $A$ as
$\lambda_1(A) \geq \lambda_2(A) \geq \cdots \geq \lambda_n(A),$
where $\lambda_1(A)$ and $\lambda_n(A)$ denote the maximum and minimum eigenvalues of $A$, respectively.
   \end{itemize}
   
\item Functional analysis:
    \begin{itemize}
        \item For $1\leq p < \infty$, and $\nu \in \mP(S),$ we denote by $
L^p(\nu)$ the Banach space defined by  $$
L^p(\nu)= \left\{\,(f\colon S \mapsto \mbR) : \|f\|_{L^p(\nu)} := \left(\int_S |f(s)|^p\, \nu(\sd s)\right)^{1/p} < \infty \right\}.$$
        \item In particular, for $p=2$, the inner product in the Hilbert space $L^2(\nu)$ is denoted by $$\langle \varphi, \psi \rangle_{L^2(\nu)}:=\int_S\varphi(s) \psi(s)  \nu(\sd s).$$
        \item For $p=\infty$, we define
\[
L^\infty(\nu)
:=
\left\{(f:S\to\mbR) \;:\;
\|f\|_{L^\infty(\nu)}
:=
\operatorname{ess\,sup}_\nu |f|
<\infty
\right\},
\]
where 
\[
\operatorname{ess\,sup}_\nu |f|
:=
\inf\Big\{
M\in\mbR \;:\;
\nu(\{s\in S:\,|f(s)|>M\})=0
\Big\}.
\] Intuitively, $\operatorname{ess\,sup}_\nu |f|$ is the smallest $M$
such that $|f(s)|\le M$ for $\nu$-almost every $s\in S$.

    \end{itemize}

\end{itemize}

\section{Importance samplers} \label{sec:sIS2}


\subsection{Model} \label{ssec:BI}
Let $(\Omega,\Sigma,\mbP)$ be a probability space and let
$X:\Omega\mapsto\mX \subseteq \Real^{d_x}$ and
$Y:\Omega\mapsto\mY \subseteq \Real^{d_y}$
denote two multidimensional r.v.s. To be specific, the r.v. $X$ represents the
$d_x$-dimensional signal of interest (or state), which we aim to estimate, and
$Y$ represents some $d_y$-dimensional observed data.

The complete model can be specified as follows:
\begin{itemize}
\item The prior probability law of the state of interest $X$ is denoted by
$\pi_0(\mathrm{d}x)$, and we assume the ability to generate random samples from
this distribution.

\item Given $X = x$, the observation $Y$ admits a conditional probability
density function (pdf) w.r.t. the Lebesgue measure, denoted by $g(y \mid x)$.
\end{itemize}

We introduce the notation \( g_y^{d_x}(x) := g(y \mid x) \) and define an
associated likelihood function 
\beq\label{eq_ly}
l_y^{d_x}(x) := \mathtt{c}\, g_y^{d_x}(x),
\eeq
where \( \mathtt{c} \in \mbR^+ \) is a positive (possibly unknown) constant,
independent of \( x \).
The superscript $d_x$ explicitly denotes the dependence of the
conditional pdf $g(y \mid x)$ on the dimension of $X$. Similar superscript notation is used for other dimension-dependent objects
throughout this paper. \cred{}

We refer to a pair $\mathcal{M}^{d_x} = (\pi_0, g^{d_x})$ as a model for Bayesian inference. Throughout this work, we consider a family of models $\{\mathcal{M}^{d_x}\}_{d_x \in \mathbb{N}}$, indexed by the dimension $d_x$ of the state variable $X$. Intuitively, the models in the family $\{\mathcal{M}^{d_x}\}_{d_x \in \mathbb{N}}$ may be expected to be ``of the same type'', i.e., similar in structure and with differences due mainly to the increasing dimension $d_x$. An explicit example (of a family of linear and Gaussian models) is given in Section~\ref{sec:Linear_Gaussian}. Nevertheless, the conditions required by the main theorems in Sections \ref{sec:General_Random} and \ref{sec:Special_cases} actually allow for families where the individual models can be rather different. The dimension of the observation space, $d_y$, is assumed to be fixed unless stated otherwise.

Usually, the aim is to approximate the posterior law of the state $X$ conditional on the observation $Y=y$, denoted $\pi_y(\sd x)$, using IS. This posterior probability measure can be written as
\beq
\pi_y(\sd x) =\frac{
	g_y^{d_x}(x) \pi_0(\sd x)
}{
	\pi_0(g_y^{d_x})
}= \frac{
	l_y^{d_x}(x) \pi_0(\sd x)
}{
	\pi_0(l_y^{d_x})
},
\label{eqDef_piy}
\eeq
where 
\beq
\pi_0(g_y^{d_x}):=\int g_y^{d_x}(x) \pi_0(\sd x)
\nn
\eeq
is the normalization constant of the posterior distribution for a fixed
observation $y\in\mY$, and coincides with the pdf of
the r.v. $Y$ when $y\in\mY$ is regarded as the independent variable of the function $y \leadsto \pi_0(g_y^{d_x})$.

We remark that $\pi_0(g_y^{d_x})$ is analytically intractable in general. One important exception is the case of linear Gaussian models, which we explicitly discuss in Section~\ref{sec:Linear_Gaussian}.

We assume that $g_y^{d_x}(x) \in L^2(\pi_0)$, $\forall\; y\in\mY$, i.e.,
\begin{equation}\label{eq:normL2}
\| g_y^{d_x} \|^2_{L^2(\pi_0)} :=
\int_{\mX} g_y^{d_x}(x)^2\, \pi_0(\sd x) < \infty.
\end{equation}
This condition holds if and only if the normalized importance sampling weights have finite variance under the proposal distribution $\pi_0$ (see Appendix \ref{app:Z_ratio}). In particular, it is trivially satisfied whenever $g_y^{d_x}(x)$ is uniformly bounded for all
$x\in\mX, \, y \in \mY$.

The approximation of the posterior law $\pi_y(\sd x)$ for a given (fixed)
observation $Y=y$ can be carried out using a simple IS scheme as shown in
Section~\ref{Algorithm}.


\subsection{Importance sampler}\label{Algorithm}

Assume a model $\mM^{d_x}$ as described in Section~\ref{ssec:BI}. To describe a general IS scheme for the approximation of $\pi_y(\sd x)$, let us introduce an importance probability measure $\nu$ such that  $\pi_0$ is absolutely continuous w.r.t. $\nu$ (denoted $\pi_0 \ll \nu$).  
In this setting, we let
\[
\rho(x) := \frac{\sd \pi_0}{\sd \nu}(x), \qquad x \in \mX,
\]
denote the Radon--Nikodym derivative of $\pi_0$ w.r.t. $\nu$. A general IS procedure, with proposal $\nu$, is outlined in Algorithm \ref{A1} below.

\begin{algorithm}\label{A1}
A general importance sampler
\begin{enumerate}
\item Draw $N$ iid samples $x^1, \ldots, x^N$ with common probability law
$\nu$.
\item Compute the likelihood values $l_y^{d_x}(x^i)$ for
$i=1,\ldots,N$.
\item Compute normalised importance weights
$w_\nu^i = \frac{ l_y^{d_x}(x^i)\rho(x^i) }{ \sum_{j=1}^N l_y^{d_x}(x^j)\rho(x^j) }$, $i=1, ..., N$.
\end{enumerate}
\end{algorithm}

The samples $x^1, \ldots, x^N$ and their importance weights yield a Monte Carlo
estimator of the posterior measure $\pi_y(\sd x)$, namely
\beq
\pi_{y,\nu}^{N}(\sd x) := \sum_{i=1}^N w_\nu^i \delta_{x^i}(\sd x),
\nn
\eeq
where $\delta_{x^i}(\sd x)$ is the Dirac delta measure located at $x^i$.

For clarity, we carry out our analysis for a simpler version of this general
algorithm where the prior is used as the proposal, i.e. $\nu=\pi_0$. The resulting scheme is displayed as
Algorithm~\ref{A2}. Hereafter, we refer to this specific procedure as standard
IS.

\begin{algorithm}\label{A2}
Importance sampler
\begin{enumerate}
\item Draw $N$ iid samples $x^1, \ldots, x^N$ from the prior law $\pi_0$.
\item Compute the likelihood values $l_y^{d_x}(x^i)$ for
$i=1,\ldots,N$.
\item Compute normalised importance weights
$w^i = \frac{l_y^{d_x}(x^i)}{\sum_{j=1}^N l_y^{d_x}(x^j)}$.
\end{enumerate}
\end{algorithm}

The standard importance sampler outputs the random probability measure
\beq
\pi_y^{N}(\sd x) = \sum_{i=1}^N w^i \delta_{x^i}(\sd x).
\label{eqDef_piyN}
\eeq
This measure can be easily used to approximate posterior expectations of $X$, i.e., integrals w.r.t. $\pi_y(\sd x)$. For a given real test function
$f:\mX \mapsto \Real$, let us denote
\beq
\pi_y(f) := \int_\mX f(x)\pi_y(\sd x)
=
\mbE[f(X) \mid Y=y],
\nn
\eeq
i.e., $\pi_y(f)$ is the expected value of the r.v. $f(X)$ conditional on $Y=y$.
We can naturally approximate
\beq
\pi_y(f) \approx \pi_y^{N}(f)
= \sum_{i=1}^N f(x^i) w^i,
\nn
\eeq
and then analyze the random errors $\pi_y(f)-\pi_y^{N}(f)$.

\begin{remark}
There is no loss of generality in analysing the Algorithm~\ref{A2} instead of the general importance sampler outlined in Algorithm~\ref{A1}. As discussed in detail in Appendix~\ref{sec:importance_sampling}, any general importance sampler with proposal \(\nu\) for a model $\mM^{d_x}=(\pi_0,g^{d_x})$ can be interpreted as a standard algorithm for a modified model \(\widetilde{\mM}^{d_x}=(\nu,\widetilde g^{d_x} )\), with prior \(\nu\) and likelihood
\[
\widetilde g_y^{d_x}(x) := g_y^{d_x}(x)\, \frac{\sd \pi_0}{\sd \nu}(x).
\]  
Consequently, it suffices to analyze the performance of Algorithm~\ref{A2} applied to \(\widetilde{\mM}^{d_x}\). 
\end{remark}

\begin{remark}
    It is pertinent to observe that the likelihood function $l_y^{d_x}(x) \propto g_y^{d_x}(x)$ serves as the foundation for the algorithmic approximation $\pi_y^{N}$ of the posterior measure $\pi_y$. While this likelihood appears pervasively throughout the subsequent analysis, the density $g_y^{d_x}(x)$ remains the primary object of interest. Notably, since the IS approximation $\pi_y^N$ is constructed via self-normalized weights, the resulting measure is invariant to the choice of the proportionality constant in $l_y^{d_x}(x)$.
\end{remark}


\subsection{Error bounds}\label{sec:EB}

We aim to obtain an upper bound for the  approximation error
\beq\label{eq:norm_Exp}
\| \pi_y(f) - \pi_y^{N}(f) \|_p:=\mbE \left[ \big| \pi_y(f) - \pi_y^{N}(f) \big|^p \right]^{1/p},
\eeq 
where the expectation is taken w.r.t. the probability law of the random samples $x^1, \ldots, x^N$.

Under mild assumptions on the model defined in Section~\ref{ssec:BI}, as specified in Theorem~\ref{thIS2}, we derive an explicit error bound for the estimator $\pi_y^{N}(f)$ when $f \in B(\mX)$.

\begin{theorem}\label{thIS2}
Fix $y \in \mathcal{Y}$, and assume that $g_y^{d_x}(x)$ is strictly
positive and bounded, i.e.,
\[
g_y^{d_x}(x) > 0 \quad \text{for all } x \in \mX,
\quad \text{and} \quad
\|g_y^{d_x}\|_{\infty} < \infty.
\]
Then, for any $f \in B(\mX)$ and $p \geq 1$, there exists a constant $C_y^{d_x} < \infty$, independent of $N$ and $f$, such that   
\begin{equation}\label{eq:ThIS2}
\| \pi_y(f) - \pi_y^{N}(f) \|_p
\leq
\frac{C_y^{d_x} \, \|f\|_{\infty}}{\sqrt{N}}.
\end{equation}
\end{theorem}

\noindent\textbf{Proof:} See Appendix~\ref{appThIS2}. \qed

The inequality in \eqref{eq:ThIS2} is well known \cite{agapiou2017importance}; it states that the approximation error rate is $\mathcal{O}(N^{-\frac{1}{2}})$. Moreover, our analysis provides an explicit upper bound for the constant $C_y^{d_x}$ in \eqref{eq:ThIS2}, namely
\begin{equation}\label{ineq_C_ydx}
C_y^{d_x} \leq \frac{B_p \|g_y^{d_x}\|_{\infty}}{\pi_0(g_y^{d_x})},
\end{equation}
where $B_p$ is a finite constant depending solely on $p$ (refer to \eqref{eq:C_ydz} in Appendix \ref{appThIS2}). This constant arises from the Marcinkiewicz--Zygmund inequality and is independent of the dimensions $d_x$ and $d_y$.

\begin{remark}
The pdf of the r.v.\ \(Y\), is precisely the normalizing constant \(\pi_0(g_y^{d_x})\) that appears explicitly in the upper bound of \eqref{ineq_C_ydx}. The ratio \(\|g_y^{d_x}\|_{\infty} / \pi_0(g_y^{d_x})\) captures the possible dependence of this bound on the state dimension \(d_x\). A closely related quantity, obtained by replacing \(\|g_y^{d_x}\|_{\infty}\) with \(\|g_y^{d_x}\|_{L^2(\pi_0)}\), plays a central role in the subsequent analysis, where \(Y\) is treated as a r.v. This is the main subject of Section \ref{sec:General_Random}.
\end{remark}


\section[Models with random observations]{Random observations}
\label{sec:General_Random}

In this section, we adopt the viewpoint that $Y$ itself is a r.v. generated by the model $\mM^{d_x}$ instead of using it as given (fixed) data. With this perspective, the likelihood function, the normalization constant and the posterior measure also become random objects. In particular, the quantity
\(
  \big| \pi_Y(f) - \pi_Y^{N}(f) \big|
\)
becomes a $Y$-dependent random variable. The subscript $_Y$ emphasizes the randomness of the observation, the same notation is used hereafter for other observation-dependent objects.
To derive an upper bound on this random error, we take the expectation
w.r.t. the law of $Y$ in order to obtain an averaged error.
This leads us naturally to examine how integration over the random
observation is performed.
The central tool for this development is the Bochner integrability
\cite{dashti2015bayesian} of a specially constructed function, which acts
as a link between the fixed and random observation formulations.

\subsection{Random posterior measure}

We  adopt the perspective that the observation $Y$ is a r.v. generated by the model in Section~\ref{ssec:BI}. In that setting,
\beq
l_Y^{d_x}(x) :=  \mathtt{c}\, g(Y \mid x) \quad \text{and} \quad \pi_0(l_Y^{d_x})
\nn
\eeq
are real-valued r.v.s (recall that $\mathtt{c} \in \mbR^+$ is an arbitrary positive constant). Consequently, the posterior distribution becomes a random probability
measure, which we denote as
\[
\pi_Y(\sd x) =\frac{l_Y^{d_x}(x)\,\pi_0(\sd x)}{\pi_0(l_Y^{d_x})}=\frac{g_Y^{d_x}(x)\,\pi_0(\sd x)}{\pi_0(g_Y^{d_x})}.
\]
Recall that  \( g_y^{d_x}(x) = g(y \mid x) \) coincides with the
conditional pdf of \( Y \) given \( X = x \).
As a consequence, \( g_y^{d_x}(x) \) is strictly positive, that is,
\[
g_y^{d_x}(x) > 0, \quad \forall (x,y) \in \mX \times \mY,
\]
and it is normalized, i.e.,
\begin{equation*}
    \int_{\mathcal{Y}} g_y^{d_x}(x)\,\sd y = 1,
    \quad \forall x \in \mX,
\end{equation*}
where \( \sd y \) denotes the Lebesgue measure on \( \mathcal{Y} \).
Finally, we recall the assumption that \( g_y^{d_x}(x) \in L^2(\pi_0) \) for all \( y \in \mathcal{Y} \). It is clear that Algorithm \ref{A2} yields an estimator $\pi_Y^N$ of $\pi_Y$, using a likelihood function of the form $l_Y^{d_x}(x) = \mathtt{c}\, g_y^{d_x}(x)$ for an arbitrary constant $\mathtt{c} > 0$ independent of $x$. Since the algorithm relies on normalized weights, the particular value of this constant has no effect on the resulting estimator. The key difference is that the estimator \( \pi_Y^{N} \) now depends on the realizations of the random variable \( Y \). Consequently, $\pi_Y^{N}$ is itself a random probability measure, whose randomness arises both from the Monte Carlo samples $X=x^i$, $i = 1,\ldots,N$, and from the observation $Y$.

\noindent 
We aim to derive an upper bound for the approximation error in \eqref{eq:norm_Exp}, averaged over the observations, where the expectation is taken w.r.t. the joint law of the random samples $x^1, \ldots, x^N$ and the observation $Y$. Much of the analysis to be elaborated in this section depends on two
objects that we introduce next:
\begin{itemize}
\item The probability measure \( \eta \in \mP(\mY) \), defined as
\beq\label{eq:eta}
\eta(\sd y) := \pi_0(g_y^{d_x})\,\sd y,
\eeq
which corresponds to the marginal probability law of the random variable $Y$ induced by the model $\mathcal{M}^{d_x}$ and also serves as the normalization constant in Eq.~\eqref{eqDef_piy}.

\item The measurable map
\beq
\begin{array}{llll}
    \ell^{\,d_x}: &\mathcal{Y} &\mapsto &L^2(\pi_0)\\
    &y &\leadsto &\ell_y^{\,d_x} := \dfrac{g_y^{d_x}(\cdot)}{\pi_0(g_y^{d_x})}= \dfrac{l_y^{d_x}(\cdot)}{\pi_0(l_y^{d_x})},
\end{array}
\label{eq:link}
\eeq
hereafter termed the \emph{link function}, which assigns to each observation
\( y \) a normalized density over \( \mX \). Since \( \pi_0(g_y^{d_x}) > 0, \; \forall y\in \mY\), the map \( \ell^{\,d_x}:\mY \mapsto L^2(\pi_0) \) is well defined.
\end{itemize}
For any $p \in \mathbb{N}$, the $L^2(\pi_0)$-norm of $\ell_y^{\,d_x}$ can be written as
\begin{equation}
\label{norm_link_f}
\|\ell_y^{\,d_x}\|_{L^2(\pi_0)}^p
=
\frac{\|g_y^{d_x}\|_{L^2(\pi_0)}^p}{\bigl(\pi_0(g_y^{d_x})\bigr)^p}.
\end{equation}
For any test function \( \varphi \in L^2(\pi_0) \), the posterior expectation conditioned on the random observation \( Y \) is given by
\begin{equation*}
\pi_Y(\varphi) = \int_\mX \varphi(x) \, \pi_Y (\sd x)=\mathbb{E}\!\left[ \varphi(X) \mid \mathcal{G} \right],
\end{equation*}
where \( \mathcal{G} := \sigma(Y) \) denotes the sigma-algebra generated by $Y$. Moreover, this quantity admits the representation
\begin{equation}\label{eq:pi_Y_Inner}
\pi_Y(\varphi)
= \frac{\int_\mX \varphi(x)\,g_y^{d_x}(x)\,\pi_0(\sd x)}
       {\pi_0(g_y^{d_x})}
= \langle \varphi, \ell_Y^{\,d_x} \rangle_{L^2(\pi_0)},
\end{equation}
where \( \langle \cdot,\cdot \rangle_{L^2(\pi_0)} \) denotes the inner product
in the Hilbert space \( L^2(\pi_0) \).

\subsection{Bochner integrability}
\label{sec:Bochner_Int}

Take test functions $\varphi, \psi \in L^2(\pi_0)$. 
The quantity $\mathbb{E}[\pi_Y(\varphi)]$ defines a linear functional on $L^2(\pi_0)$, 
while $\mathbb{E}[\pi_Y^2(\varphi)]$ defines a quadratic form induced by the symmetric bilinear form
\[
\begin{array}{rcll}
\mathrm{B}: & L^2(\pi_0) \times L^2(\pi_0) & \mapsto & \mathbb{R},\\[1mm]
& (\varphi, \psi) & \leadsto & \mathbb{E}[\pi_Y(\varphi)\,\pi_Y(\psi)],
\end{array}
\]
where the expectation is taken w.r.t. the law of the random observation $Y$.

\noindent By the inner-product representation in
Eq.~\eqref{eq:pi_Y_Inner}, the boundedness of
\( \mbE[\pi_Y^p(\varphi)] \), \( p\in\{1,2\} \), is closely related to the
Bochner integrability of the map \( \ell^{\,d_x}: y\leadsto \ell_y^{\,d_x} \). This connection is made precise below:

\noindent The function \( \ell^{\,d_x} \) is called Bochner $p$-integrable if and only if it is strongly \( \eta \)-measurable \cite{tuomas2016analysis}, and
\begin{equation}
\label{eq:Kp}
\int_{\mathcal{Y}} \|\ell_y^{\,d_x}\|_{L^2(\pi_0)}^p\,\eta(\sd y) =: K_p^{d_x} < \infty.
\end{equation}
The strong $\eta$-measurability condition is satisfied under mild regularity assumptions on \( g_y^{d_x}(x)\), which already hold for the models of interest (see Lemma~\ref{lem:S_M} in Appendix \ref{apx:Bochner}). Hence, in the remainder of the paper we only require the finiteness of
\( K_p^{\,d_x} \) to guarantee the Bochner $p$-integrability of $\ell^{\, d_x}$.

For \( 1 \leq p \leq \infty \), we denote by
\( L^p(\mathcal{Y}; L^2(\pi_0)) \)
the Banach space of Bochner \( p \)-integrable functions
\( \upsilon : \mathcal{Y} \mapsto L^2(\pi_0) \)
w.r.t. the measure \( \eta \) (cf. \cite{tuomas2016analysis}).
In particular, for \( p = \infty \), the norm is defined as
\[
\|\upsilon\|_{L^\infty(\mathcal{Y}; L^2(\pi_0))}
:=
\operatorname*{ess\,sup}_{y \in \mathcal{Y}}
\| \upsilon_y \|_{L^2(\pi_0)} < \infty.
\]
Since \( \eta \) is a probability measure, for \( 1 \leq p < \infty \),
we have the continuous embedding
\[
L^{\infty}(\mathcal{Y}; L^2(\pi_0))
\subseteq
L^{p+1}(\mathcal{Y}; L^2(\pi_0))
\subseteq
L^p(\mathcal{Y}; L^2(\pi_0)).
\]

With this notation, we formalize the above discussion in the following 

\begin{theorem}\label{thm:4.1}
Assume that $\ell^{\,d_x}$ is Bochner $p$-integrable
for some $p \in \{1,2\}$. Then, for all $\varphi \in L^2(\pi_0)$,
\[
\mathbb{E}\!\left[\pi_Y^p(\varphi)\right]
\le
K_p^{d_x}\,\|\varphi\|_{L^2(\pi_0)}^{p},
\]
where $K_p^{d_x}<\infty$ is defined in \eqref{eq:Kp}.
\end{theorem}
\begin{proof}
See Lemma~\ref{lem:bounded_functional} and Remark~\ref{rem:proof_cor} in Appendix \ref{apx:Bochner}.
\end{proof}

\subsection{Approximation errors}

We now provide a convergence result for the IS Monte Carlo approximation of the posterior
expectation with random observations.

\begin{theorem}\label{thm:thISR}
Assume that $g_y^{d_x}(x) \in L^{2}(\pi_0)$ for all $y \in \mathcal{Y}$.
Then, for any $f \in B(\mX)$, the following statements are equivalent:

\begin{enumerate}[label=\roman*)]
    \item There is a constant \( C_2^{d_x} < \infty \), independent of \( N \) and
    \( f \), such that 
    \beq\label{eq:Iff_error}
    \| \pi_Y(f) - \pi_Y^{N}(f) \|_2
    \leq
    \frac{C_2^{d_x} \|f\|_\infty}{\sqrt{N}}.
    \eeq
    \item  The link function $\ell^{\, d_x}$ is Bochner square-integrable, i.e.
\[
\ell^{\, d_x}\in L^2(\mathcal Y;L^2(\pi_0)).
\]
\end{enumerate}
Moreover, $C_2^{d_x}$ can be decomposed as
\[
C_2^{d_x} = B_2\, K_2^{d_x},
\]
where $B_2$ is a constant independent of the state dimension $d_x$ and $d_y$ arising from the
Marcinkiewicz--Zygmund inequality and
$K_2^{d_x}$ accounts for the possible dependence on $d_x$ through the
Bochner integrability of the link function $\ell^{\, d_x}$.
\end{theorem}

\begin{proof}
See Appendix~\ref{appThIS3}.
\end{proof}

Theorem~\ref{thm:thISR} provides a necessary and sufficient condition for the
$L^2$ approximation error of the IS estimator $\pi_Y^N$,
produced by Algorithm~\ref{A2}, to converge at the canonical rate
$\mathcal{O}(N^{-1/2})$. The result holds in the setting where the observation
$Y$ is treated as a r.v. and yields a complete characterization of
the $L^2$ Monte Carlo error of the posterior approximation: the rate
holds if and only if the link function is
square integrable in the Bochner space $L^2(\mathcal{Y};L^2(\pi_0))$.
In particular, the convergence of the random-observation Monte Carlo approximation is
characterized by this intrinsic integrability condition on the link
function, which separates the statistical properties of the observation model
from those of the sampling scheme. Moreover, this characterization provides
a direct mechanism to control the dependence of the error bounds on the state
dimension $d_x$, a feature that is made explicit in the following theorem.

\begin{theorem}\label{thIS2R}
Assume that $\ell^{\, d_x}$ is Bochner $p$-integrable
for some $p \in \{1,2\}$. Suppose further that there exists a polynomial
$P_{n,p}(d_x)$ of degree $n$, such that
\begin{equation}\label{eq:K_p_leq_P}
K_p^{d_x} \le P_{n,p}(d_x), \qquad \forall d_x \in \mathbb{N}.
\end{equation}
Then, for any $f \in B(\mX)$,
\[
\| \pi_Y(f) - \pi_Y^{N}(f) \|_p
\le
\frac{B_p P_{n,p}(d_x)\,\|f\|_\infty}{\sqrt{N}},
\]
where $B_p$ is a constant arising from the Marcinkiewicz--Zygmund inequality and depends only on $p$.
\end{theorem}

\begin{proof}
See Appendix \ref{ap:T4.3}.
\end{proof}

\begin{remark}\label{rem:sample_complexity}
The error bounds established in Theorem~\ref{thIS2R} demonstrate how specific structural properties of the model can mitigate the curse of dimensionality in importance samplers. Specifically, under these assumptions, the number of samples required to approximate the posterior expectation $\pi_Y(f)$ with a prescribed average error increases at most polynomially with the dimension $d_x$, rather than exponentially.
From a practical perspective, this provides a direct prescription for the sample complexity required to maintain a prescribed accuracy. Given a target tolerance $\varepsilon > 0$ such that  $\|\pi_Y(f) - \pi_Y^N(f)\|_p \le \varepsilon$, the sample size $N$ should be tuned relative to the state dimension according to:
\begin{equation*}
N \;\ge\; \left( \frac{P_{n,p}(d_x) \|f\|_\infty}{\varepsilon} \right)^2.
\end{equation*}
In models where the observation map scales linearly with dimension (i.e., $n=1$), a quadratic increase in $N$ is sufficient to preserve the fidelity of the approximation. We provide an explicit construction of such a system, demonstrating this manageable polynomial scaling, in Section~\ref{sub:Uniformly_Bounded_Observation}.
\end{remark}

\begin{corollary}\label{cor:IS_uniform}
Under the assumptions of Theorem~\ref{thIS2R}, suppose that \( n = 0 \), i.e., there
exists a constant \( \mathcal{K}_p < \infty \), independent of \( d_x \), such
that
\begin{equation}\label{eq:U_Kdz}
\sup_{d_x \in \mathbb{N}} K_p^{d_x} \leq \mathcal{K}_p,
\end{equation}
then, for any \( f \in B(\mX) \), there exists a constant
\( C_p < \infty \), independent of \( N \), \( f \), and \( d_x \),
such that 
\beq
\| \pi_Y(f) - \pi_Y^{N}(f) \|_p
\leq
\frac{C_p \, \|f\|_\infty}{\sqrt{N}},
\qquad p \in \{1,2\}.
\nn
\eeq
In particular,
\[
\lim_{N \to \infty}
\| \pi_Y(f) - \pi_Y^{N}(f) \|_p = 0,
\]
uniformly w.r.t. the state dimension \( d_x \).
\end{corollary}

\begin{remark}\label{rem:C_p_BK} Corollary \ref{cor:IS_uniform} is a straightforward consequence of Theorem \ref{thIS2R}.
The constant \( C_p < \infty \) admits the decomposition
\[
C_p = B_p \, \mathcal{K}_p,
\]
where \( B_p \) is a constant depending only on \( p \), arising from the
Marcinkiewicz--Zygmund inequality.
\end{remark}

\section{Special cases}\label{sec:Special_cases}

In this section, we examine specific scenarios under which the assumptions of Theorem~\ref{thIS2R} and Corollary~\ref{cor:IS_uniform} are satisfied, that is, when the inequalities \eqref{eq:K_p_leq_P} and \eqref{eq:U_Kdz} holds. It is worth emphasizing that the bounds derived here are not intended to be tight. 
Instead, our goal is to show that, under the theory developed in the Section \ref{sec:General_Random}, the error bounds for the importance sampling estimator \(\pi_Y^N\) may scale polynomially, or even remain uniformly bounded, w.r.t. the state dimension \(d_x\).

\noindent Throughout this Section, we focus on the case \(p=2\), which subsumes the case \(p = 1\)

\begin{remark}\label{rem:link_norm}
Observe that the quantity
\[
\| \ell_y^{\, d_x} \|_{L^2(\pi_0)}^2
=
\frac{\pi_0([g_y^{d_x}]^2)}{\pi_0(g_y^{d_x})^2}
\]
is directly related to the $\chi^2$-divergence between the posterior
and the prior distributions. It is also related to the
relative variance of the Monte Carlo estimator of the normalization
constant $Z := \pi_0(g_y^{d_x})$ when sampling from the prior
(see Appendix~\ref{app:Z_ratio}) and has already been found to be relevant in the analysis of IS methods; see, e.g., \cite{agapiou2017importance,akyildiz2021convergence}. 
\end{remark}

In the setting of the present work, where the observation $Y$ is assumed random, we can use Eq. \eqref{norm_link_f} and the identity \( \eta(\mathrm{d}y) = \pi_0(g_y)\,\sd y \) to obtain
\begin{equation*}
K_p^{d_x}= \int_{\mathcal{Y}} \| \ell_y^{\,d_x} \|_{L^2(\pi_0)}^p \, \eta(\mathrm{d}y)
=
\int_{\mathcal{Y}} \frac{ \| g_y^{d_x} \|_{L^2(\pi_0)}^p }{ \pi_0(g_y^{d_x})^{p-1} } \, \sd y.
\end{equation*}
In particular, for the case $p=2$ we have
\beq\label{eq:ratio}
K_2^{d_x}=
\int_{\mX} \left( \int_{\mathcal{Y}} \frac{g_y^{d_x}(x)^2}{\pi_0(g_y^{d_x})} \, \mathrm{d}y \right) \pi_0(\mathrm{d}x).
\eeq
If \(\ell^{\, d_x}_Y\) is regarded as a random element in \(L^2(\pi_0)\), the Bochner integrability condition in \eqref{eq:Kp} is equivalent to the moment condition
\begin{equation}
\label{eq:BI_as_Ex}
K_p^{d_x} = \mathbb{E}\!\left[\| \ell_Y^{\,d_x} \|_{L^2(\pi_0)}^p\right] < \infty .
\end{equation}

The previous discussion shows that the key condition required for the
dimension-dependent error bounds of Theorem~\ref{thIS2R} to hold reduces to the
finiteness of the moment in \eqref{eq:BI_as_Ex}. In other words, the error
of the IS estimator is governed by the integrability of the link function
$\ell_Y^{d_x}$ as a random element of $L^2(\pi_0)$.
In the remainder of this section, we verify this condition for several
relevant observation models. The analysis illustrates how the structure
of the likelihood determines the dependence of the error bounds on the
state dimension $d_x$, and shows that polynomial, or even uniform,
bounds can arise in common statistical models. We begin with the linear
Gaussian model, since in this case all the relevant distributions can be
computed analytically, regardless of the dimensionality of the problem.

\subsection{Linear Gaussian model}\label{sec:Linear_Gaussian}

We consider a family of linear Gaussian models $\{\mM^{d_x}\}_{d_x \in \mbN}$, indexed by the state dimension \(d_x\). Let \(X \sim \pi_0:=\mathcal{N}(\mu_x^{\,d_x},\Sigma_x^{\,d_x})\) be a \(d_x\)-dimensional Gaussian random vector with mean \(\mu_x^{\,d_x} \in \mbR^{d_x}\) and positive definite covariance matrix \(\Sigma_x^{\,d_x} \in \mbR^{d_x \times d_x}\).  
The observation model is defined as  
\begin{equation*}
Y = A^{d_x}X + V, 
\end{equation*}
where \(A^{d_x} \in \mbR^{d_y \times d_x}\) is a fixed matrix and \(V \sim \mathcal{N}(0,R)\) is a zero-mean Gaussian noise vector with positive definite covariance matrix \(R \in \mbR^{d_y \times d_y}\), independent of \(X\) and also independent of its dimension $d_x$.
As a result, the r.v. \(Y\) is Gaussian with marginal law
\[
Y \sim \mathcal{N}(\sd y; \mu_y^{d_x},\Sigma_y^{d_x}),
\]
where
\begin{equation}
\mu_y^{d_x} = A^{d_x}\mu_x^{d_x}, 
\qquad 
\Sigma_y^{d_x} = A^{d_x}\Sigma_x^{\,d_x} [A^{d_x}]^\top + R.
\label{eq:mu_Sd}
\end{equation}
In particular,
\beq
\pi_0(g_y^{d_x})=\mN( y; \mu_y^{d_x}, \Sigma_y^{d_x} ),
\label{eqPdf-y}
\eeq
i.e., $\pi_0(g_y^{d_x})$ is the pdf of $Y$.

Hereafter we show that, for the linear Gaussian model described above, the Bochner constant \(K_2^{d_x}\) can exhibit either polynomial growth or remain uniformly bounded w.r.t. the state dimension \(d_x\), depending on the behavior of the model parameters. In particular, we derive sufficient conditions under which \(K_2^{d_x}\) satisfies a polynomial bound in \(d_x\), as well as stronger conditions ensuring the existence of a finite constant \(\mathcal K_2\), independent of \(d_x\), such that inequality \eqref{eq:U_Kdz} holds. Consequently, the assumptions of Corollary~\ref{cor:IS_uniform} are satisfied. 

\subsubsection{Analysis of the Bochner constant \texorpdfstring{$K_2^{d_x}$}{K2^{dx}}}

Let us note that, since \(A^{d_x}\Sigma_x^{d_x}[A^{d_x}]^\top\) is positive semidefinite and \(R\) is positive definite, it follows that \(\Sigma_y^{d_x}\) in \eqref{eq:mu_Sd} is also positive definite.
In addition, the likelihood function $g_y^{d_x}(x)$ corresponds to the Gaussian pdf of $Y$ conditional on $X=x$, namely,
\[
g_y^{d_x}(x) = \mathcal{N}(y; A^{d_x}x, R).
\]
Squaring this pdf yields
\beq\label{eq:g_y-square}
g_y^{d_x}(x)^2
= \frac{1}{(2\uppi)^{d_y} |R|}
\exp\!\left(
-\frac{1}{2}(y-A^{d_x}x)^\top (2R^{-1})(y-A^{d_x}x)
\right),
\eeq
which is proportional to a Gaussian density with covariance \(\tfrac{1}{2}R\). Indeed, we can express
\beq\label{eq:C_1}
g_y^{d_x}(x)^2 = C_1~\mathcal{N}\!\left(y; A^{d_x}x, \tfrac{1}{2}R\right)
\eeq
for some constant $C_1$. On the other hand,
\beq\label{eq:Normal_1R2}
\mathcal{N}\!\left(y; A^{d_x}x, \tfrac{1}{2}R\right)
=
(2\uppi)^{-\frac{d_y}{2}} |\tfrac{R}{2}|^{-\frac{1}{2}}
\exp\!\left(-\frac{1}{2}(y-A^{d_x}x)^\top (2R^{-1})(y-A^{d_x}x)\right).
\eeq
Combining \eqref{eq:g_y-square}-\eqref{eq:Normal_1R2}, and simplifying terms, yields the constant
\[
C_1
=
(4\pi)^{-d_y/2} |R|^{-1/2}.
\]
Consequently, recalling that $\pi_0(\mathrm{d}x) = \mN(\mathrm{d}x; \mu_x^{d_x}, \Sigma_x^{d_x})$, we readily obtain
\beq
\int_\mX g_y^{d_x}(x)^2\,\pi_0(\mathrm{d}x)
=
\frac{1}{(4\uppi)^{d_y/2} |R|^{1/2}}
\mathcal{N}(y;\mu_y^{d_x},\Sigma_y^{d_x} - \tfrac{1}{2}R).
\label{eqPdf-y2}
\eeq
Using \eqref{eqPdf-y} and \eqref{eqPdf-y2}, we realise that, in order to satisfy the integrability condition \eqref{eq:ratio}, it suffices to find a finite bound for the expectation
\begin{equation}
\label{eq:norm_ell_Gauss}
\mathbb{E}\!\left[\| \ell_Y^{\,d_x} \|_{L^2(\pi_0)}^2\right]
=
\frac{1}{(4\uppi)^{d_y/2} |R|^{1/2}}
\int_\mY
\frac{\mathcal{N}(y;\mu_y^{d_x},\Sigma_y^{d_x}-\tfrac{1}{2}R)}
{\mathcal{N}(y;\mu_y^{d_x},\Sigma_y^{d_x})}
\,\sd y.
\end{equation}
Let \(S_1 := \Sigma_y^{d_x}\) and \(S_2 := \Sigma_y^{d_x} - \tfrac{1}{2}R\). The integrand in \eqref{eq:norm_ell_Gauss} can be rewritten as
\[
\frac{\mathcal{N}(y;\mu_y^{d_x},S_2)}{\mathcal{N}(y;\mu_y^{d_x},S_1)}
=
\left(\frac{|S_1|}{|S_2|}\right)^{1/2}
\exp\!\left(
-\tfrac{1}{2}(y-\mu_y^{d_x})^\top (S_2^{-1}-S_1^{-1})(y-\mu_y^{d_x})
\right).
\]
Moreover, both \( S_1 \) and \( S_2 \) are positive definite, and
\[
S_1 = \Sigma_y^{d_x} \succ \Sigma_y^{d_x} - \tfrac{1}{2}R = S_2,
\]
where \( \succ \) denotes the Loewner order (see Section 7.7 in ~\cite{horn2012matrix}). Let us also define
\beq 
\label{eq_SS2S1}
S^{-1} := S_2^{-1} - S_1^{-1}.
\eeq
By the operator monotonicity of matrix inversion on the cone of positive definite matrices (see Theorem 24 in \cite{magnus2019matrix}), it follows that
\[
S_1 \succ S_2 \quad \Rightarrow \quad S_1^{-1} \prec S_2^{-1} \quad \Rightarrow \quad S^{-1} \succ 0
\]
(hence, $S^{-1}$ is positive definite). The integrand in \eqref{eq:norm_ell_Gauss} is thus proportional to the density of a Gaussian distribution with mean $\mu_y^{d_x}$ and covariance $S$. 
Then, we can write 
\beq\label{eq:Gaus_C2}
\left(\frac{|S_1|}{|S_2|}\right)^{1/2}\exp\!\left(
-\tfrac{1}{2}(y-\mu_y^{d_x})^\top (S^{-1})(y-\mu_y^{d_x})
\right)= C_2 \,\mathcal{N}(y;\mu_y^{d_x},S)
\eeq
for some constant $C_2$. On the other hand, 
\beq\label{eq:Gauss_C22}
\mathcal{N}(y;\mu_y^{d_x},S)= \frac{1}{(2\uppi)^{d_y/2} \abs{S}^{1/2}} \exp\!\left(
-\tfrac{1}{2}(y-\mu_y^{d_x})^\top (S^{-1})(y-\mu_y^{d_x})
\right),
\eeq
and combining \eqref{eq:Gaus_C2} and \eqref{eq:Gauss_C22} we obtain
\beq\label{eq:C2}
C_2=\left(\frac{ (2\uppi)^{d_y} \abs{S} |S_1|}{|S_2|}\right)^{1/2}.
\eeq
Therefore, the integral in \eqref{eq:norm_ell_Gauss} can be computed exactly to yield
\beq\label{eq:int_g2_g1_Gauss}
\int_\mY \,
\frac{\mathcal{N}(y;\mu_y^{d_x},S_2)}{\mathcal{N}(y;\mu_y^{d_x},S_1)}
\,\sd y = C_2  \int_\mY \,\mathcal{N}(y;\mu_y^{d_x},S) \, \sd y
=
C_2.
\eeq
Using the identity $S_2^{-1}-S_1^{-1} = S_1^{-1}(S_1-S_2)S_2^{-1}$, and since $S_1-S_2=\cfrac{1}{2}R$,
Eq. \eqref{eq_SS2S1} readily yields
\beq
S^{-1} = S_1^{-1}\left(\cfrac{1}{2}R\right)S_2^{-1} 
\nn
\eeq
and, as a consequence,
\[
S = S_2(2R^{-1})S_1
\]
and 
\beq
\abs{S} = 2^{d_y} \frac{\abs{S_1}\abs{S_2}}{\abs{R}}.
\label{eq__b}
\eeq
Substituting  \eqref{eq__b} into \eqref{eq:C2} we arrive at
\beq
C_2 = (4\uppi)^{d_y/2}\frac{   |\Sigma_y^{d_x}|}{|R|^{1/2}}
\label{eq__c}
\eeq
and then, combining \eqref{eq__c}, \eqref{eq:int_g2_g1_Gauss} and \eqref{eq:norm_ell_Gauss}, we readily see that
\begin{equation}
K_2^{d_x}=\label{eq:B_LG}
\mathbb{E}\!\left[\| \ell_Y^{\,d_x} \|_{L^2(\pi_0)}^2\right]
=\frac{   |\Sigma_y^{d_x}|}{|R|}
,
\end{equation}
which is always finite (provided that $R \succ 0$, which implies $\Sigma_y^{d_x} \succ 0$ as well), ensuring that the moment condition \eqref{eq:BI_as_Ex} holds and so does the Bochner integrability condition \eqref{eq:Kp}.
By \eqref{eq:mu_Sd} and \eqref{eq:B_LG} the Bochner constant is
\begin{equation}\label{eq:Boch_Det}
K_2^{d_x}
=
\frac{
\left|
A^{d_x}\Sigma_x^{d_x}(A^{d_x})^\top + R
\right|
}{
|R|
}.
\end{equation}
Since \(R\) is positive definite, we can rewrite this expression as
\begin{equation*}
K_2^{d_x}
=
\left|
I_{d_y}
+
R^{-1/2}
A^{d_x}\Sigma_x^{d_x}(A^{d_x})^\top
R^{-1/2}
\right|.
\end{equation*}

\subsubsection{Dependence of \texorpdfstring{$K_2^{d_x}$}{K2(dx)} on the dimensions \texorpdfstring{$d_x$}{dx} and \texorpdfstring{$d_y$}{dy}}

Let us define the matrix
\begin{equation*}
F^{d_x}
=
R^{-1/2}
A^{d_x}\Sigma_x^{d_x}(A^{d_x})^\top
R^{-1/2}.
\end{equation*}
Then,
\begin{equation*}
K_2^{d_x}
=
|I_{d_y}+F^{d_x}|.
\end{equation*}
Since \(F^{d_x}\) is positive semidefinite, all its eigenvalues satisfy
\[
\lambda_i(d_x,d_y)\ge 0,
\qquad
i=1,\ldots,d_y,
\]
where the eigenvalues depend on both the state dimension \(d_x\) and the observation dimension \(d_y\). Hence, we obtain
\begin{equation}\label{eq:BK_lam}
K_2^{d_x}
=
\prod_{j=1}^{d_y}
\left(
1+\lambda_j(d_x,d_y)
\right).
\end{equation}
Moreover, the quantity in \eqref{eq:BK_lam} can be bounded in terms of the model parameters. Since \(\Sigma_x^{d_x}\) is positive semidefinite, we may write
\begin{equation*}
F^{d_x}
=
G^{d_x}
(G^{d_x})^\top,
\end{equation*}
where
\begin{equation*}
G^{d_x}
=
R^{-1/2}
A^{d_x}
(\Sigma_x^{d_x})^{1/2}
\in
\mathbb R^{d_y\times d_x}.
\end{equation*}
Therefore,
\begin{equation*}
\lambda_i(d_x,d_y)
=
\sigma_i^2
(
G^{d_x}
),
\qquad
i=1,\ldots,d_y.
\end{equation*}
Consequently, the eigenvalues depend on the model matrices through the singular values of the noise-normalized and prior-covariance-weighted observation matrix
\[
R^{-1/2}
A^{d_x}
(\Sigma_x^{d_x})^{1/2}.
\]
We can therefore write
\begin{equation}\label{K_2:sigma}
K_2^{d_x}
=
\prod_{i=1}^{d_y}
\left(
1+\sigma_i^2
(
G^{d_x}
)
\right).
\end{equation}

\begin{remark}
Assume that \(\sigma_i(G^{d_x}) \neq 0\) for all \(i=1,\ldots,d_y\). Taking the limit as \(d_y\to\infty\), we obtain
\begin{equation*}
\lim_{d_y\to \infty} K_2^{d_x}
=
\prod_{i=1}^{\infty}
\left(
1+\sigma_i^2
(
G^{d_x}
)
\right).
\end{equation*}
By Theorem 8.52 in \cite{apostol1958mathematical}, this infinite product converges if and only if
\begin{equation*}
\sum_{i=1}^{\infty}
\sigma_i^2
(
G^{d_x}
)
<
\infty.
\end{equation*}
Therefore, the Bochner constant is not necessarily unbounded as the observation dimension \(d_y\to\infty\). A uniform error bound may still be obtained provided the singular values associated with the model parameters are suitably controlled.
\end{remark}

We now return to the setting in which \(d_y\) is fixed. Using the singular-value inequality
\begin{equation*}
\sigma_i(BC)
\le
\sigma_1(B)\sigma_i(C),
\end{equation*}
which holds for any pair of compatible matrices \(B\) and \(C\), we obtain
\begin{align*}
\sigma_i^2
(
G^{d_x}
)
&\le
\lambda_1(R^{-1})
\sigma_i^2(A^{d_x})
\lambda_1(\Sigma_x^{d_x}).
\end{align*}
Hence,
\begin{equation}\label{bound51}
\sigma_i^2
(G^{d_x})
\le
\frac{
\lambda_{1}(\Sigma_x^{d_x})
}{
\lambda_{d_y}(R)
}
\,
\sigma_i^2(A^{d_x}).
\end{equation}
Substituting \eqref{bound51} into \eqref{K_2:sigma} yields
\begin{equation}\label{K_2Up_Bound}
K_2^{d_x}
\le
\prod_{i=1}^{d_y}
\left(
1+
\frac{
\lambda_{1}(\Sigma_x^{d_x})
}{
\lambda_{d_y}(R)
}
\sigma_i^2(A^{d_x})
\right)
\le
\left(
1+
\frac{
\lambda_{1}(\Sigma_x^{d_x})
}{
\lambda_{d_y}(R)
}
\sigma_1^2(A^{d_x})
\right)^{d_y}.
\end{equation}
Therefore, the dependence of the Bochner constant on the state dimension \(d_x\) is determined by the growth of the maximal eigenvalue of \(\Sigma_x^{d_x}\) and the maximal singular value of \(A^{d_x}\).

\begin{remark}
Assume that there exist polynomials \(P_{n_A}(d_x)\) and \(P_{n_\Sigma}(d_x)\), of degrees \(n_A,n_\Sigma\in\mathbb N\), such that
\[
\sigma_1(A^{d_x}) \le P_{n_A}(d_x),
\qquad
\lambda_1(\Sigma_x^{d_x}) \le P_{n_\Sigma}(d_x).
\]
Then,
\[
K_2^{d_x} \le P_n(d_x),
\]
where \(P_n\) is a polynomial of degree
\[
n=(2n_A+n_\Sigma)d_y.
\]
Hence, if the maximal singular value of the observation operator and the maximal eigenvalue of the prior covariance matrix grow at most polynomially in \(d_x\), then the Bochner constant \(K_2^{d_x}\) also grows at most polynomially satisfying the assumptions in Theorem \ref{thIS2R}.

Moreover, if
\beq\label{Unif_Cst_LG}
C_A:=\sup_{d_x \in \mathbb{N}} \sigma_1(A^{d_x}) < \infty,
\qquad
C_\Sigma:=\sup_{d_x \in \mathbb{N}} \lambda_1(\Sigma_x^{d_x}) < \infty,
\eeq
then there exists a constant \(\mathcal{K}_2 < \infty\), independent of \(d_x\), such that
\[
K_2^{d_x} \le \mathcal{K}_2,
\]
with
\beq\label{eq:uniform_LG}
\mathcal{K}_2
:=
\left(
1 +
\frac{C_\Sigma \, C_A^2}{\lambda_{d_y}(R)}
\right)^{d_y}.
\eeq
In particular, under these uniform boundedness assumptions, the conditions of Corollary~\ref{cor:IS_uniform} are satisfied.
A numerical example illustrating polynomial and uniform error bounds w.r.t. the state dimension \(d_x\) is provided in Appendix~\ref{ap:Numerical_Example}.
\end{remark}

As discussed above, the key quantity governing the dependence of
the error bounds on the state dimension $d_x$ is the moment condition
\eqref{eq:BI_as_Ex}. However, verifying this condition directly is generally difficult in
nonlinear or non-Gaussian models, since it involves high-dimensional
integrals. In the following subsection, we show that, under suitable pointwise
domination assumptions on the likelihood, the condition can instead be
verified through simple bounds. These assumptions yield explicit
estimates for the second moment of the link function  required in
Theorem~\ref{thIS2R}.


\subsection{Pointwise Bounded Likelihood}\label{sub:BI_Noise}
Assume that the likelihood function $g_y^{d_x}(x)$ admits a product-form pointwise bound 
\begin{equation}\label{gy_hk}
g_y^{d_x}(x) \le m^{d_x}(x) q(y),
\end{equation}
where $m^{d_x}: \mX \to [0, \infty)$ and $q: \mathcal{Y} \to [0, \infty)$ are measurable functions satisfying
\begin{equation}\label{eq__HK}
\|m^{d_x}\|_{\infty} \le M(d_x) < \infty \quad \text{and} \quad \|\, q \,\|_{L^1(\mathcal{Y})} \le Q < \infty,
\end{equation}
for constants $M(d_x)$ and $Q$. Specifically note that $\|m^{d_x}\|_\infty$ may depend on the dimension $d_x$ but the upper bound $Q$ does not. Under conditions \eqref{gy_hk}, we readily see from Remark \ref{rem:link_norm} that  the squared norm of the link function $\ell_y^{\, d_x}$ satisfies
\begin{equation}\label{eq__l}
\| \ell_y^{\, d_x} \|_{L^2(\pi_0)}^2 \le q(y) \frac{ \int_{\mX} m^{d_x}(x) g_y^{d_x}(x) \pi_0(\mathrm{d}x) }{ \left[ \int_{\mX} g_y^{d_x}(x) \pi_0(\mathrm{d}x) \right]^2 }
\end{equation}
for any $y \in \mathcal{Y}$. By treating the observation $Y$ as a r.v. and taking the expectation on both sides of \eqref{eq__l}, we obtain
\begin{align*}
\mathbb{E} \left[ \| \ell_Y^{\, d_x} \|_{L^2(\pi_0)}^2 \right] 
&\le \int_{\mathcal{Y}} q(y) \frac{ \int_{\mX} m^{d_x}(x) g_y^{d_x}(x) \pi_0(\mathrm{d}x) }{ \int_{\mX} g_y^{d_x}(x) \pi_0(\mathrm{d}x) } \, \mathrm{d}y \\
&= \int_{\mathcal{Y}} q(y) \pi_y(m^{d_x}) \, \mathrm{d}y \\
&\le M(d_x) \int_{\mathcal{Y}} q(y) \, \mathrm{d}y = M(d_x) Q =: K_2^{d_x} < \infty,
\end{align*}
where the first inequality follows from the construction of the marginal law of $Y$ (see Eq. \eqref{eq:eta}) and then we simply apply the bounds $\|m^{d_x}\|_{\infty} \leq M(d_x)$ in \eqref{eq__HK}.

Consequently, if \eqref{gy_hk} and \eqref{eq__HK} hold, the second moment is finite, implying that the link function $\ell^{\, d_x} \in L^2(\mathcal{Y}; L^2(\pi_0))$. Furthermore, if there exists a uniform constant $M$ such that  $M(d_x) \le M$ for all $d_x \in \mathbb{N}$, then the requirements of Corollary~\ref{cor:IS_uniform} are satisfied for the class of models described in  Section \ref{sub:BI_Noise}.

The second-moment condition \eqref{eq:BI_as_Ex} generally involves
high-dimensional integrals, which may be difficult to verify directly in
practical models. In the following subsection, we bridge the gap between
this abstract functional-analytic condition and concrete Bayesian
modeling by deriving a tractable criterion for the integrability of the
link function. 
Specifically, we consider likelihoods belonging to the class of
elliptically symmetric families \cite{delmas2024elliptically}, which
include a wide range of noise models such as Gaussian, Laplace,
Student-$t$, and Cauchy distributions. For this class of models, we show
that, when the observation function is bounded, the Bochner
integrability condition reduces to the finiteness of a one-dimensional
radial integral.
This reduction enables explicit analytical constructions showing that
the approximation error grows at most polynomially with the state
dimension $d_x$. In particular, we identify scenarios in which the
polynomial degree is zero, yielding error bounds that hold uniformly in
$d_x$.


\subsection[Uniformly bounded observation function]{Uniformly bounded observation function}\label{sub:Uniformly_Bounded_Observation}

Let us consider the observation model
\begin{equation*}
Y = h^{d_x}(X) + V,
\end{equation*}
where $h^{dx}:\mX\to\mY$ is a (generally nonlinear) deterministic function
and $V$ is an observation noise term modeling sensor and measurement errors.
We assume that $V$ is independent of $X$ and admits a density on $\mY$.
Conditionally on $X=x$, the observation $Y$ then admits a density
\(
g^{d_x}(y\mid x).
\)

\noindent
A broad class of observation likelihoods can be represented in the
\emph{radial--shift form}
\begin{equation}
\label{eq:general-radial-likelihood-fixed}
g^{d_x}(y\mid x)
=
\phi\!\big(\psi(\|y-h^{d_x}(x)\|_{R})\big),
\end{equation}
where $\phi:\mbR\to(0,\infty)$ is non-increasing,
$\psi:(0,\infty)\to\mbR$ is non-decreasing, and
$\|\cdot\|_{R}$ denotes a Mahalanobis norm on $\mY$ defined below.

The discrepancy between an observation $y$ and its state-dependent
prediction $h^{d_x}(x)$ is measured through the quadratic form
\begin{equation*}
(y-h^{d_x}(x))^\top R (y-h^{d_x}(x)),
\end{equation*}
where $R\in\mbR^{d_y\times d_y}$ is a fixed symmetric positive definite
matrix independent of the state. This quadratic form induces the
Mahalanobis inner product
\[
\langle y_1,y_2\rangle_R := y_1^\top R y_2,
\]
and the associated norm
\[
\|y\|_R := \sqrt{\langle y,y\rangle_R},
\qquad y\in\mbR^{d_y}.
\]

Representation \eqref{eq:general-radial-likelihood-fixed} separates two
modeling components:
\begin{itemize}
\item the \emph{geometric structure}, determined by the Mahalanobis metric
induced by $R$ together with the translation $h^{d_x}(x)$, and
\item the \emph{radial profile} $\phi\circ\psi$, which governs the tail
behaviour of the likelihood.
\end{itemize}

Conditional densities of the form
\eqref{eq:general-radial-likelihood-fixed} yield the class
of \emph{elliptically symmetric likelihoods}.

\begin{remark}\label{rem:Colour_noises}
 All the noise models listed in Tables A and B  fit into the elliptically symmetric  likelihood family describe by \eqref{eq:general-radial-likelihood-fixed}.

\noindent \begin{table}[H]
\centering
\begin{tabular}{ll}
\hline
\textbf{Noise model} & \textbf{Radial deformation $\psi(r)$} \\
\hline
Gaussian 
& $\tfrac12\,r^2$ \\[0.4em]

Generalized Gaussian / exponential power 
& $r^\beta,\qquad \beta>0$ \\[0.4em]

Laplace 
& $r$ \\[0.4em]

Sub-Gaussian 
& $a\,r^2,\qquad a>0$ \\[0.4em]
\hline
\end{tabular}
\caption{\textbf{Table A.} Exponential--type radial profile: 
$\phi(s)=\mathrm{C}\,\mathrm{e}^{-s}$, where $\mathrm{C}$ is a normalization constant.}
\label{tab:radial_profiles}
\end{table}

\begin{table}[h]
\centering
\begin{tabular}{lll}
\hline
\textbf{Noise model} 
& \textbf{Radial deformation $\psi(r)$} 
& \textbf{Tail parameter $\alpha$} \\
\hline
Student--$t$ $(\nu>0)$ 
& $1+\tfrac{1}{\nu}r^2$ 
& $\tfrac{\nu+d_y}{2}$ \\[0.4em]

Cauchy 
& $1+r^2$ 
& $\tfrac{d_y+1}{2}$ \\[0.4em]

Pearson type VII $(\lambda>0)$
& $1+\tfrac{1}{\lambda}r^2$ 
& $\alpha>0$ \\[0.4em]

Generalized Cauchy / rational quadratic 
& $1+r^p,\qquad p\in\mathbb{N}$ 
& $\alpha>0$ \\[0.4em]
\hline
\end{tabular}
\caption{\textbf{Table B.} Polynomial (heavy--tailed) radial profile: 
$\phi(s)=\mathrm{C}\,s^{-\alpha}$, where $\mathrm{C}$ is a normalization constant.}
\label{tab:polynomial_profiles}
\end{table}
\end{remark}

\noindent 
Consider a family of models $\{\mathcal{M}^{d_x}\}_{d_x \in \mathbb{N}}$ indexed by the state dimension $d_x$, where each model is of the form $\mathcal{M}^{d_x}=\{\pi_0, g_y^{d_x}\}$ and $g_y^{d_x}(x)$ belongs to the class of elliptically symmetric likelihoods given by \eqref{eq:general-radial-likelihood-fixed}. Hereafter, we assume that the observation map $h^{d_x} : \mX\subseteq \mbR^{d_x} \to \mY$ is
uniformly bounded  w.r.t. the Euclidean norm on $\mY$ by a function $M:\mbN \mapsto [0,\infty)$, of the state dimension $d_x$, namely
\begin{equation}
\label{eq:MR-def}
M(d_x):= \sup_{x \in \mX} \|h^{d_x}(x)\| < \infty.
\end{equation}
Under this assumption, it is straightforward to show  (see Appendix~\ref{sub:Mahalanobis-geo}) that 
\begin{equation*}
M_R^{d_x} := \sup_{x \in \mX} \|h^{d_x}(x)\|_R \le \sqrt{\lambda_1(R)}\, M(d_x).
\end{equation*}
Let us denote
\begin{equation*}
\mathcal{S}_R^{d_y} :=
|R|^{1/2}\, \mathbb{S}^{d_y-1} =
|R|^{1/2}\, \frac{2 \pi^{d_y/2}}{\Gamma(d_y/2)},
\end{equation*}
$\mathbb{S}^{d_y-1}$ is the surface area of the $d_y$--unit sphere in
$\mbR^{d_y}$, and $\Gamma(\cdot)$ is the Gamma function. 

\noindent The following result reduces the high-dimensional Bochner integrability condition for the link function \(\ell^{\,d_x}\), given in \eqref{eq:BI_as_Ex}, to the finiteness of a one-dimensional radial integral.

\begin{theorem}\label{thm:Colour_Noises}
Let the observation likelihood $g^{d_x}(y\mid x)$ belong to the class of elliptically symmetric  families defined by \eqref{eq:general-radial-likelihood-fixed}. Assume that the observation map $h^{d_x}:\mX\mapsto\mathcal{Y}$ satisfies \eqref{eq:MR-def}. Then the second moment of the link function $\ell_Y^{\, d_x}$ is  bounded by a one-dimensional radial integral, namely,
\begin{equation}
\label{eq:final_uniform_integral}
\mathbb{E} \left[ \| \ell_Y^{\, d_x} \|_{L^2(\pi_0)}^2 \right] \le \mathcal{S}_R^{d_y} \int_0^\infty \frac{\phi^2 \big( \psi(r) \big)}{\phi \big( \psi(r + 2M_R^{d_x}) \big)} \, r^{d_y-1} \, \mathrm{d}r.
\end{equation}
\end{theorem}

A proof is provided in Appendix~\ref{sec:thm:Colour_Noises} (see also Appendix \ref{sub:Mahalanobis-geo}).

\begin{corollary}\label{cor:Colour_Noises}
Let the assumptions of Theorem~\ref{thm:Colour_Noises} hold. For the observation noise models listed in Tables A and B, the integral on the right hand side of \eqref{eq:final_uniform_integral} is finite. Therefore
\[
\mathbb{E} \left[ \| \ell_Y^{\, d_x} \|_{L^2(\pi_0)}^2 \right] = K_2^{d_x} < \infty.
\]
Hence $\ell^{\, d_x}$ is Bochner $2$-integrable for this class of observation models.
Furthermore, we have the following:
\begin{enumerate}
    \item If there exists a uniform bound $M < \infty$, such that  $\sup_{d_x \in \mbN} M(d_x) \le M$, then the conditions of Corollary~\ref{cor:IS_uniform} are satisfied for all observation models in Tables A and B.
    \item For the polynomial radial profiles in Table B, if $M(d_x) \le P(d_x)$ for some polynomial $P$ of $d_x$, then $K_2^{d_x}$ satisfies \eqref{eq:K_p_leq_P}, thereby fulfilling the requirements of Theorem~\ref{thIS2R}.
\end{enumerate}
\end{corollary}

A proof is provided in Appendix~\ref{sec:cor:Colour_Noises}.

\medskip
\noindent 
\subsubsection{Example: Polynomial growth of $M(d_x)$} 
To illustrate the asymptotic behavior characterized in Corollary~\ref{cor:Colour_Noises}, we consider a high-dimensional Bayesian inverse problem on a state space $\mX \subseteq \mbR^{d_x}$. Let $\mathbf{x} = (x_1, \dots, x_{d_x})^\top \in \mX$. We define the observation map $h^{d_x}: \mX \mapsto \mathcal{Y}$ through a system of $d_y$ observation channels. For each channel $i \in \{1, \dots, d_y\}$, the measurement is given by
\begin{equation*}
    h_i^{d_x}(\mathbf{x}) = \sum_{j=1}^{d_x} a_{ij} \sigma(x_j),
\end{equation*}
where $a_{ij} \in \mbR$ are coupling coefficients and $\sigma: \mbR \to \mbR$ is a saturating nonlinearity (e.g., $\sigma(x) = \tanh(x)$ or $\sigma(x) = \operatorname{erf}(x)$). We introduce the following assumptions to characterize the operator's growth:

\begin{enumerate}
    \item Uniform local saturation: The nonlinearity $\sigma$ is globally bounded, i.e., there exists $M_\sigma > 0$, independent of the dimension $d_x$, such that  $\|\sigma\|_\infty \le M_\sigma$. 
    \item Sensitivity uniformity: The coupling coefficients are bounded such that  $|a_{ij}| \le A$ for some $A > 0$ for all $i, j$ and all $d_x$.
\end{enumerate}

The supremum of the observation magnitude, $M(d_x) := \sup_{\mathbf{x} \in \mX} \|h^{d_x}(\mathbf{x})\|$, characterizes the maximum signal energy the system captures as a function of the state dimension $d_x$. Under the assumptions of local saturation and uniform sensitivity, the forward map satisfies
\begin{equation*}
    M(d_x) \le \sqrt{d_y} A M_\sigma d_x.
\end{equation*}
Indeed, by the triangle inequality and the uniform bound on $\sigma$, the $i$-th observation channel satisfies
\begin{equation*}
    |h_i^{d_x}(\mathbf{x})| \le \sum_{j=1}^{d_x} |a_{ij}| |\sigma(x_j)| \le \sum_{j=1}^{d_x} A M_\sigma = A M_\sigma d_x.
\end{equation*}
Taking the Euclidean norm of the vector $h^{d_x}(\mathbf{x}) = (h_1^{d_x}(\mathbf{x}), \dots, h_{d_y}^{d_x}(\mathbf{x}))^\top$, we obtain
\begin{equation*}
    \|h^{d_x}(\mathbf{x})\| = \left( \sum_{i=1}^{d_y} |h_i^{d_x}(\mathbf{x})|^2 \right)^{1/2} \le \left( \sum_{i=1}^{d_y} (A M_\sigma d_x)^2 \right)^{1/2} = \sqrt{d_y} A M_\sigma d_x.
\end{equation*}
Since this bound holds for all $\mathbf{x} \in \mX$, the result follows. This formulation ensures that for any finite dimension $d_x$, the observation operator remains bounded on the domain $\mX$. Notably, the growth of the observation map satisfies $M(d_x) \le P(d_x) := a_1 d_x$ for $a_1 = \sqrt{d_y} A M_\sigma > 0$, then $P$ is a polynomial of degree 1 in the state dimension. Consequently, for any observation model $Y = h^{d_x}(X) + V$ where the noise $V$ admits a polynomial radial profile as described in Remark~\ref{rem:Colour_noises} $M(d_x)\propto d_x$ grows linearly and satisfies the conditions of Corollary~\ref{cor:Colour_Noises}~(item 2). This ensures that the requirements of Theorem~\ref{thIS2R} are fulfilled, guaranteeing that the error of the IS estimator $\pi_Y^N$ scales at most polynomially (in fact, linearly) with the state dimension $d_x$.


\section{Comparison with some existing results}\label{sec:Comparison}

The performance of IS estimators has been studied by several authors. Here we present a brief comparison of the results introduced in Sections \ref{sec:General_Random} and \ref{sec:Special_cases} with the analysis of \cite{chatterjee2018sample} and \cite{agapiou2017importance}, respectively. This is certainly not exhaustive, but we believe these references are significant and representative of the state of the art.


\subsection{Arbitrary target distribution $\pi$}

In \cite{chatterjee2018sample}, Chatterjee and Diaconis present a performance analysis of a generic importance sampler with target distribution $\pi$ and given proposal $\nu$, with $\pi \ll \nu$ and 
$$
{\rm KL}(\pi,\nu) := \int \log\left(\frac{\sd \pi}{\sd \nu}\right) \sd \pi < \infty,
$$
where KL stands for the Kullback-Leibler divergence. The estimator of the target measure $\pi$ is constructed as
$$
\widehat \pi^N(\sd x) = \sum_{i=1}^N w^i \delta_{x^i}(\sd x),
$$
where $x^i\sim\nu$ iid, for $i=1, \ldots, N$,
\beq
 w^i = \frac{
    \rho(x^i)
}{
    \sum_{j=1}^N \rho(x^j)
} \quad
\text{and}
\quad
\rho(x) \propto \frac{\sd \pi}{\sd \nu}(x).
\nn
\eeq
Under these minimal assumptions, Chatterjee and Diaconis prove that, for any test function $f \in L^2(\pi)$ and denoting $D:={\rm KL}(\pi,\nu)$,
\beq
\text{if $\log N = D + o(D)$ then}~~
\mbP\left(
    \left|
        \widehat \pi^N(f) - \pi(f)
    \right| \ge \frac{
        2\| f \|_{L^2(\pi)} \varepsilon_D
    }{
        1- \varepsilon_D
    }
\right) \le 2\varepsilon_D,
\nn
\eeq
where (under a mild assumption on the tails of $\rho(x)$) $\lim_{D\to\infty} \varepsilon_D=0$, i.e., 
$$
\text{if $\log N = D + o(D)$ then $\lim_{D\to\infty} \widehat \pi^N = \pi$ in probability.}
$$
They also show that taking $\log N = D - o(D)$ can lead to significant errors with non-negligible probability for a certain choice of $f \in L^2(\pi)$. Therefore,
\beq
\widehat\pi^N \to \pi~~\text{in probability}
~~\Leftrightarrow~~
N \asymp \exp\left\{
    {\rm KL}( \pi, \nu )    
\right\}.
\label{eqIFF}
\eeq
If the KL divergence between the target $\pi$ and the proposal $\nu$ scales linearly with the state dimension $d_x$ (as it is the case if, e.g., $\pi$ and $\nu$ have independent marginals) then expression \eqref{eqIFF} implies that
\beq
\widehat\pi^N \to \pi~~\text{in probability}
~~\Leftrightarrow~~
N \propto \exp\left\{
    d_x   
\right\}.
\label{eqIFF2}
\eeq
Expression \eqref{eqIFF2} may seem at odds with Theorem \ref{thIS2R}, which states that (under regularity assumptions),
\beq
\| \pi_Y^N(f) - \pi_Y(f) \|_2 \le \frac{P_{n,2}(d_x) \|f\|_\infty}{\sqrt{N}},
\label{eqReminder}
\eeq
where $P_{n,2}(\cdot)$ is a polynomial of degree $n$. Expressions \eqref{eqIFF2} and \eqref{eqReminder} can be reconciled if we recall that:
\begin{itemize}
\item[(i)] the target distribution $\pi_Y$ in our analysis is not arbitrary, but the posterior distribution in a Bayesian inference problem, hence $\pi_Y(\sd x) \propto g_Y^{d_x}(x)\pi_0(\sd x)$, and
\item[(ii)] $g_Y^{d_x}$ is a family of likelihood functions indexed by the random observation $Y$ (i.e., a function-valued r.v.).
\end{itemize}
The analysis of \cite{chatterjee2018sample} is carried out for a given (almost arbitrary) relative density $\rho(x)=\frac{\sd \pi}{\sd \nu}(x)$ and the goal is to determine the computational cost (value of $N$) needed to guarantee ``small errors'' without further assumptions. In this paper, instead, we work with a weight function $\propto g_Y^{d_x}(x)$ which is random and endowed with considerable structure, and the goal is to determine the regularity conditions that guarantee ``small {\em average} errors'' with a computational cost that depends at most polynomially on $d_x$. Theorem \ref{thIS2R} provides such regularity conditions in an abstract setting, while in Section \ref{sec:Special_cases} we describe a few concrete examples, i.e., specific constructions of the likelihood $g_Y^{d_x}$ that yield small average errors with non-exponential computational cost. 

\subsection{Importance sampling for Bayesian inference}

Agapiou {\em et al.}~\cite{agapiou2017importance} analyse the performance of
importance sampling  for Bayesian inference in the linear--Gaussian
in close analogy with the framework developed in Section~\ref{sec:Linear_Gaussian}. Their framework
considers a fixed observation $Y=y$ and targets the posterior
$\pi_y(\sd x)$ defined in \eqref{eqDef_piy}. The IS algorithm coincides with
Algorithm~\ref{A2}: particles are drawn independently from the prior
$x^i \sim \pi_0$, $i=1,\ldots,N$, and the posterior is approximated by
$\pi_y^N$ in \eqref{eqDef_piyN}.

\noindent A central quantity in~\cite{agapiou2017importance} is
\[
    \rho_y
    :=
    \frac{\pi_0(g_y^2)}{\pi_0(g_y)^2},
    \label{eq:def_rho_y}
\]
which is the second moment of the normalised importance weights.
Equivalently,
\[
\rho_y = 1 + \chi^2(\pi_y,\pi_0),
\]
so that $\rho_y$ measures the $\chi^2$-divergence between posterior and
prior. Using the link-function notation, by
Remark~\ref{rem:link_norm}, we have
\begin{equation*}
\rho_y
=
\|\ell_y\|_{L^2(\pi_0)}^2 .
\end{equation*}
One main result of~\cite{agapiou2017importance} yields the conditional MSE
bound
\begin{equation}
\sup_{\|\phi\|_\infty \le 1}
\mathbb{E}\!\left[
    (\pi_y(\phi)-\pi_y^N(\phi))^2
\right]
\le
\frac{4}{N}\rho_y .
\label{eq:agapiou_mse}
\end{equation}
For the linear--Gaussian model, $\rho_y$ can be related to the
\emph{intrinsic dimension}, (see Section \ref{sec:Linear_Gaussian})
\[
\tau = \mathrm{trace}(C),
\qquad
C=\sqrt{\Sigma_x^{d_x}}\,[A^{d_x}]^\top R^{-1}A^{d_x}\,\sqrt{\Sigma_x^{d_x}},
\]
whose eigenvalues $\{\lambda_{C,i}\}_{i=1}^{d_x}$ satisfy
\[
\tau=\sum_{i=1}^{d_x}\lambda_{C,i},
\qquad
\log\rho_y
=
\mathrm{cst}(y)
+
\frac12
\sum_{i=1}^{d_x}\log(1+\lambda_{C,i}).
\]
Consequently,
\[
\rho_y \gtrsim \exp\!\Big\{\tfrac12 \tau\Big\},
\]
and therefore \eqref{eq:agapiou_mse} suggest an exponential growth of the
required number of particles
\[
\mathrm{MSE} \lesssim \frac{1}{N}\exp\!\Big\{\tfrac12\tau\Big\}.
\]
To achieve $\mathrm{MSE}\le\varepsilon$, one needs
$N \gtrsim \varepsilon^{-1}\exp\{\tau\}$.
Thus the computational cost can grow exponentially with the intrinsic
dimension.

\medskip

\noindent
\emph{Relation with the present work.}
The key difference in our analysis is that the observation $Y$ is kept
random. Instead of the conditional error in \eqref{eq:agapiou_mse}, we
control the averaged quantity
\[
\big\|
    \pi_Y(f)-\pi_Y^N(f)
\big\|_2.
\]
Conditionally on $\mathcal G=\sigma(Y)$, an argument analogous to
\eqref{eq:agapiou_mse} yields
\[
\mathbb{E}\!\left[
    |\pi_Y(f)-\pi_Y^N(f)|^2
    \,\middle|\,\mathcal G
\right]
\le
\frac{C\|f\|_\infty^2}{N}\rho_Y .
\]
However, instead of estimating \(\rho_Y\) pointwise in \(Y\), we integrate with respect to the law of \(Y\) and exploit structural properties of the model. In the linear--Gaussian setting, assuming the observation dimension \(d_y\) is fixed, Corollary~\ref{cor:IS_uniform}, together with \eqref{Unif_Cst_LG} and \eqref{eq:uniform_LG}, yields the uniform bound
\[
\|
    \pi_Y(f)-\pi_Y^N(f)
\|_2
\le
\left(
1 +
\frac{C_\Sigma \, C_A^2}{\lambda_{d_y}(R)}
\right)^{d_y}
\frac{B_p\|f\|_\infty}{\sqrt N}.
\]
All constants are independent of the state dimension $d_x$, and the model
dependence enters only through the observation noise covariance $R$ and
the observation dimension $d_y$.

\medskip

Hence, while the fixed-observation analysis suggests an exponential
dependence on the intrinsic dimension, averaging over the law of $Y$
allows us to recover the $N^{-1/2}$ convergence rate uniformly in the
state dimension.

\section{Conclusions} \label{sec:Discussion}

We have investigated the approximation errors of importance samplers for Bayesian inference in high-dimensional settings. The analysis is carried out in a probabilistic framework in which the observations are treated as r.v.s, and the posterior distribution is itself a random measure. This setting allows us to exploit the probabilistic structure of the model and to formulate expected error bounds, rather than conditioning on a fixed realization of the data and analyse the worst-case scenario.

Let $\pi_Y$ and $\pi_Y^{N}$ denote the true posterior distribution of a $d_x$-dimensional r.v. $X$ given an observation $Y$ and its  IS approximation, respectively. In full generality, classical upper bounds for the $L^2$ approximation error $\|\pi_y(f) - \pi_y^{N}(f)\|_2$, for a given realization $Y=y$ and for any bounded test functions $f$, may grow exponentially with dimension, in agreement with existing results for IS and related Monte Carlo schemes \cite{Bengtsson08,Snyder08,Snyder15,agapiou2017importance,chatterjee2018sample}.

In this paper, we have introduced a refined analysis, treating observations as r.v.s, we relate the approximation error $\|\pi_Y(f) - \pi_Y^{N}(f)\|_2$ to the Bochner integrability of a suitable link function $\ell^{\, d_x}$, which generates normalized likelihoods for the model. We have shown that the $L^2$ error bounds are finite and converge at the standard Monte Carlo rate $\mO(N^{-1/2})$, for arbitrarily large dimension, if and only if the link function is Bochner integrable. This result provides a precise and explicit characterization of when high dimensionality does not lead to a breakdown of the algorithm.

Moreover, the dependence of the error bounds on the dimension $d_x$ can be quantified in terms of the moments of the link function w.r.t. the marginal distribution of the observations. This perspective enables the identification of sufficient conditions under which the $L^1$ and $L^2$ error bounds grow only polynomially with dimension, as well as conditions under which they remain uniformly bounded. These results provide a rigorous explanation of how some families of models do not suffer from a curse of dimensionality.

We have illustrated the general theory through several examples. In the linear Gaussian setting, we make explicit how the Bochner $2$-integrability constant depends on the model parameters, and how its growth w.r.t. the state dimension $d_x$ is polynomial under regularity conditions. In particular, by appropriately bounding the singular values of the observation operator and the eigenvalues of the state covariance matrix, we show that it is possible to obtain error bounds that remain uniform regardless of the state dimension~$d_x$. We have also considered models with bounded or dominated observation function, conditions that are commonly satisfied in practical applications, for instance when physical sensors operate within a finite dynamic range. In all cases, we provide explicit analytical arguments showing how the structure of the model leads to favorable dimensional scaling of the average errors w.r.t. the observation distribution.



\begin{appendix}

\section{Use of importance functions} \label{sec:importance_sampling}
In this Appendix, we analyze the connection between Algorithm~\ref{A1} and Algorithm~\ref{A2}.

Recall the model described in Section \ref{ssec:BI} with prior measure $\pi_0$ and likelihood function $g_y$, and let us refer to it as model $\mM$. Then, choose an importance function $\pi_0 \ll \nu$ as in Algorithm \ref{A1}, and define the relative density
\begin{equation}
\rho(x) := \frac{\sd\pi_0}{\sd\nu}(x).
\label{eqDefRho}
\end{equation}
Define also the modified likelihood 
\begin{equation*}
\widetilde g_y^{\,d_x}(x) := g_y^{\,d_x}(x)\,\rho(x).
\end{equation*}

Let us denote by $\widetilde \mM$ the model with prior measure $\nu(\sd x)$ and likelihood $\widetilde g_y^{\,d_x}(x)$. It is straightforward to verify that, given some fixed observation $Y=y$, $\mM$ and $\widetilde \mM$ have the same marginal likelihood. Indeed, if we define
\begin{equation}
\widetilde l_y^{\,d_x}(x) := \mathtt{c} \, \widetilde g_y^{\,d_x}(x), 
\label{eqDefModL}
\end{equation}
for some constant $\mathtt{c}\in \mbR^+$ independent of $x$, then we readily see that
\begin{eqnarray}
\nu(\widetilde l_y^{\,d_x})
&=& \mathtt{c}\int_{\mX} \widetilde g_y^{\,d_x}(x)\,\nu(\sd x) \nonumber\\
&=& \mathtt{c}\int_{\mX} g_y^{\,d_x}(x)\,\pi_0(\sd x)
= \pi_0(l_y^{\,d_x}), \nonumber
\end{eqnarray}
where the second equality follows from the definition in \eqref{eqDefRho}. This observation is important, as it shows that the marginal density of $Y$ can equivalently be expressed as $\nu(\widetilde{l}_y^{\,d_x})$. Consequently, the analysis involving random observations carries over unchanged under the modified model.

Moreover, by a similar argument, we see that the modified model $\widetilde \mM$ also displays the same posterior law as the original model $\mM$, i.e., for any integrable test function $f:\mX \mapsto \mbR$ we obtain
\begin{equation}
\widetilde \pi_y(f)
= \frac{\nu(\widetilde l_y^{\,d_x} f)}{\nu(\widetilde l_y^{\,d_x})}
= \frac{\pi_0(l_y^{\,d_x} f)}{\pi_0(l_y^{\,d_x})}
= \pi_y(f).
\nonumber
\end{equation}
Therefore, model $\widetilde \mM$ can be interpreted as a reparametrisation of the original model $\mM$.

Finally, the standard importance sampling Algorithm \ref{A2} applied to model $\widetilde \mM$ yields the following steps:
\begin{enumerate}
\item Draw $x^1, \ldots, x^N$ iid from $\nu(\sd x)$.
\item Compute the likelihoods $\widetilde l_y^{\,d_x}(x^i) = \mathtt{c}\, \widetilde g_y^{\,d_x}(x^i)$, $i=1,\ldots,N$.
\item Compute the weights $w^i \propto \widetilde l_y^{\,d_x}(x^i)$, $i=1,\ldots,N$.
\end{enumerate}
Using the definitions in \eqref{eqDefRho}--\eqref{eqDefModL}, it is straightforward to verify that the procedure described above coincides exactly with Algorithm~\ref{A1}. Consequently, to assess the performance of Algorithm~\ref{A1} for the model $\mM$, it suffices to analyze the performance of Algorithm~\ref{A2} for the modified model $\widetilde{\mM}$, that is, to study the Bochner integrability of the associated link function $\widetilde{\ell^{\, d_x}}$ under $\widetilde{\mM}$. To this end, it is also necessary to ensure that the likelihood function $\widetilde g_y^{\,d_x}$ of the modified model $\widetilde{\mM}$ satisfies $\widetilde g_y^{\,d_x} \in L^2(\nu)$ for all $y\in \mY$.


\section{Statistical analysis of the normalization constant}\label{app:Z_ratio}
In the context of the posterior measure $\pi_y^{N}$, the normalization
constant $Z$ (also known as the \emph{evidence} or \emph{marginal likelihood})
is defined as the expectation of the likelihood function $g_y^{d_x}$
under the prior measure $\pi_0$,
\begin{equation*}
    Z := \int_{\mathcal{X}} g_y^{d_x}(x)\,\pi_0(\sd x)
    = \pi_0(g_y^{d_x}).
\end{equation*}
Given a set of $N$ iid samples $\{x_i\}_{i=1}^N$ drawn from the prior
distribution $\pi_0$, an unbiased Monte Carlo estimator of $Z$ is given by
the empirical mean
\begin{equation*}
    \hat{Z}_N = \frac{1}{N}\sum_{i=1}^{N} g_y^{d_x}(x_i).
\end{equation*}
This corresponds to a particular instance of the IS
framework in which the proposal distribution $\nu$ coincides with the prior
$\pi_0$. In this case, the estimator reduces to standard Monte Carlo
integration under $\pi_0$.
\subsubsection*{Variance analysis and $L^2(\pi_0)$ convergence}
To quantify the concentration of the estimator $\hat{Z}_N$ around the true
value $Z$, we analyze its variance. By independence of the samples and
linearity of expectation, the variance of $\hat{Z}_N$ satisfies
\begin{equation*}
    \mathrm{Var}(\hat{Z}_N)
    = \frac{1}{N}
      \left(
      \mathbb{E}_{\pi_0}\!\left[(g_y^{d_x}(X))^2\right]
      -
      (\mathbb{E}_{\pi_0}[g_y^{d_x}(X)])^2
      \right).
\end{equation*}
Expressing this relation in terms of the $L^2(\pi_0)$ norm yields
\begin{equation*}
    \mathrm{Var}(\hat{Z}_N)
    =
    \frac{1}{N}
    \left(
    \|g_y^{d_x}\|_{L^2(\pi_0)}^2 - Z^2
    \right).
\end{equation*}
Hence, the stability of the Monte Carlo approximation is governed by the
second moment of the likelihood function under the prior distribution.
A useful quantity to assess the magnitude of fluctuations is the
\emph{relative variance}
\begin{equation*}
    \frac{\mathrm{Var}(\hat{Z}_N)}{Z^2}
    =
    \frac{1}{N}
    \left(
\frac{\|g_y^{d_x}\|_{L^2(\pi_0)}^2}{Z^2} - 1
    \right).
\end{equation*}
Using the link function notation introduced earlier and
Remark~\ref{rem:link_norm}, we have
\[
\|\ell_y^{d_x}\|_{L^2(\pi_0)}^2
=
\frac{\|g_y^{d_x}\|_{L^2(\pi_0)}^2}{Z^2}.
\]
Therefore, the relative variance can be written as
\[
\frac{\mathrm{Var}(\hat{Z}_N)}{Z^2}
=
\frac{1}{N}
\left(
\|\ell_y^{d_x}\|_{L^2(\pi_0)}^2 - 1
\right).
\]
This expression highlights that the efficiency of the estimator is
directly controlled by the $L^2(\pi_0)$ norm of the link function.


\section{Proof of Theorem \ref{thIS2}} \label{appThIS2}

\subsection{Preliminaries}
Our goal is to establish explicit upper bounds for the \( L_p \) norms of the approximation errors
\[
\| \pi_y(f)-\pi_y^{N}(f)  \|_p, \quad p \ge 1,
\]
for any test function \( f \in B(\mX) \).

The following two propositions are straightforward to prove, hence we simply state them.

\begin{proposition}\label{MIP}
Let $\alpha_0, \beta_0 \in \mathcal{P}(S)$ and let $l, \tilde{l}$ be positive, real-valued functions on $S$ satisfying $\alpha_0(l) < \infty$ and $\beta_0(\tilde{l}) < \infty$. The probability measures $\alpha, \beta \in \mP(S)$ defined by
\[
\alpha(f) = \frac{\alpha_0(l f)}{\alpha_0(l)}, \quad \text{and} \quad
\beta(f) = \frac{\beta_0(\tilde{l} f)}{\beta_0(\tilde{l})},
\quad \forall f\in \mB(S),
\]
satisfy
\[
\left| \alpha(f) - \beta(f) \right|
\leq \frac{1}{\alpha_0(l)}
\left[
\left| \alpha_0(l f) - \beta_0(\tilde{l} f) \right|
+ \| f \|_{\infty}
\left| \alpha_0(l) - \beta_0(\tilde{l}) \right|
\right].
\]
\end{proposition}

\begin{proposition}\label{prop:Unbiased_E}
Let \( l_y^{\,d_x} \) be defined as in Eq.~\eqref{eq_ly} in Section~\ref{ssec:BI}. For any bounded measurable function \( f \), we have
\[
\mathbb{E}[\pi_0^N(f \, l_y^{\,d_x})] = \pi_0(f \, l_y^{\,d_x}).
\]
\end{proposition}
Next, we obtain a partial characterization of the approximation errors.

\begin{lemma}\label{lem:L2bound}
Let $y\in\mY$ and \( p \in \{ 1, 2\} \) be fixed. Assume that \( g_y^{d_x} \in L^2(\pi_0) \) and \( l_y^{\,d_x}(x) \) defined as in Eq.~\eqref{eq_ly}. Then, for any \( f \in \mB(\mX) \), there exists a constant \( B_p > 0 \) depending only on \( p \), such that 
\begin{equation}\label{eq:E_y_fix}
\mathbb{E}\left[
\left| \pi_0(f l_y^{d_x}) - \pi_0^N(f l_y^{d_x}) \right|^p
\right]
\leq
\frac{\mathtt{c}^p \,B_p \|f\|^p_\infty}{N^{p/2}} \,
\|g_y^{d_x}\|_{L^2(\pi_0)}^p.
\end{equation}
\end{lemma}

\begin{proof}
Fix $y \in \mathcal{Y}$ and $p \in \{1,2\}$, and let $\{x^i\}_{i=1}^N$ be an iid\ sequence of r.v.s with common distribution $\pi_0$.
Define the zero-mean iid\ r.v.s
\[
U_{i,f}
:=
\frac{1}{N}
\left(
\pi_0(f l_y^{d_x}) - f(x^i) l_y^{d_x}(x^i)
\right).
\]
Note that
\[
\sum_{i=1}^N U_{i,f}
=
\pi_0(f l_y^{d_x}) - \pi_0^N(f l_y^{d_x}).
\]
Using the inequality
\[
|x-y|^p \le 2^{p-1} (|x|^p + |y|^p),
\]
we obtain
\[
U_{i,f}^2
\le
\frac{2 \|f\|_\infty^2}{N^2}
\left(
\pi_0^2(l_y^{d_x}) + l_y^{d_x}(x^i)^2
\right).
\]
Proceeding as in the standard Marcinkiewicz--Zygmund inequality argument,
we obtain
\[
\sum_{i=1}^N
\mathbb{E}[U_{i,f}^2]
\le
\frac{2\|f\|_\infty^2}{N}
\left(
\pi_0^2(l_y^{d_x}) + \pi_0([l_y^{d_x}]^2)
\right).
\]
Since
\[
0 < \pi_0^2(l_y^{d_x})
\le \pi_0([l_y^{d_x}]^2)
= \mathtt{c}^2\|g_y^{d_x}\|_{L^2(\pi_0)}^2 < \infty,
\]
we conclude that
\[
\sum_{i=1}^N
\mathbb{E}[U_{i,f}^2]
\le
\frac{(2 \mathtt{c})^2\|f\|_\infty^2}{N}
\|g_y^{d_x}\|_{L^2(\pi_0)}^2.
\]
Applying the Marcinkiewicz--Zygmund inequality yields
\eqref{eq:E_y_fix}.
\end{proof}

\begin{corollary}\label{cor:L2bound}
Suppose that \( \|g_y^{d_x}\|_{\infty} < \infty \).
Then the conclusion of Lemma~\ref{lem:L2bound} holds for all \( p \ge 1 \), with
\[
\mathbb{E}\left[
\left| \pi_0(f l_y^{d_x}) - \pi_0^N(f l_y^{d_x}) \right|^p
\right]
\le
\frac{\mathtt{c}^p\, B_p \|f\|_\infty^p}{N^{p/2}}
\|g_y^{d_x}\|_\infty^p.
\]
\end{corollary}

\begin{proof}
The proof follows the same steps as in Lemma~\ref{lem:L2bound}, replacing
the $L^2$ bound with the uniform bound, since
\[
\|g_y^{d_x}\|_{L^2(\pi_0)} \le \|g_y^{d_x}\|_\infty.
\]
\end{proof}

\subsection{Proof of Theorem \ref{thIS2}}

We recall the definition of $l_y^{\, d_x}(x)$ in~\eqref{eq_ly} and note that
\[
\pi_y(f) - \pi_y^{N}(f)
=
\frac{\pi_0(f l_y^{d_x})}{\pi_0(l_y^{d_x})}
-
\frac{\pi_0^N(f l_y^{d_x})}{\pi_0^N(l_y^{d_x})}.
\]
Applying Proposition~\ref{MIP}, we obtain
\begin{equation*}
\left| \pi_y(f) - \pi_y^{N}(f) \right|
\le
\frac{1}{\mathtt{c} \, \pi_0(g_y^{d_x})}
\left[
\left| \pi_0^N(f l_y^{d_x}) - \pi_0(f l_y^{d_x}) \right|
+
\|f\|_\infty
\left| \pi_0(l_y^{d_x}) - \pi_0^N(l_y^{d_x}) \right|
\right].
\end{equation*}
Since the second term corresponds to the first with \( f \equiv 1 \),
it suffices to control the first one. Rasing to the power $p$, using the elementary inequality $\abs{x+y}^p \leq 2^{p-1} (\abs{x}^p+ \abs{y}^p)$ 
and  Corollary~\ref{cor:L2bound}, we obtain
\[
\| \pi_y(f) - \pi_y^{N}(f) \|_p^p
\le
\left(
\frac{B_p}{\pi_0(g_y^{d_x})}
\right)^p
\frac{\|f\|_\infty^p}{N^{p/2}}
\|g_y^{d_x}\|_\infty^p,
\]
where \( B_p < \infty \) depends only on \( p \).
This yields \eqref{eq:ThIS2} with
\beq\label{eq:C_ydz}
C_y
:=
\frac{B_p \|g_y^{d_x}\|_\infty}{\pi_0(g_y^{d_x})} < \infty.
\eeq
\qed


\section{Bochner integrability and expectation functionals
}\sectionmark{Bochner integrability \& expectation}\label{apx:Bochner}

In this appendix, we investigate the relationship between the Bochner integrability of the link function
\(\ell^{\, d_x}\), defined in Eq.~\eqref{eq:link}, and the functional and bilinear form introduced in Section~\ref{sec:Bochner_Int} for the posterior random measure \(\pi_Y\).

For the map \(\ell^{\, d_x}\) to be Bochner integrable, it is necessary and sufficient that it is strongly \(\eta\)-measurable and that the integral of its norms is finite; that is,
\[
\int_{\mathcal{Y}} \|\ell^{\, d_x}_y\|_{L^2(\pi_0)} \, \eta(\mathrm{d}y) < \infty.
\]
The following lemma guarantees strong measurability of \(\ell^{\, d_x}\), as defined in \eqref{eq:link}, for all models considered in this work.

\subsection[Bochner measurability of \(\ell^{\, d_x}\)]{Bochner measurability of \(\ell^{\, d_x}\)}\label{Apx:B_M}

\begin{lemma}\label{lem:S_M}
Let \( (\mathcal{Y}, \mathcal{B}(\mathcal{Y}), \eta) \) and
\( (\mX, \mathcal{B}(\mX), \pi_0) \) be measure spaces,
and let \( L^2(\pi_0) \) denote the Hilbert space of square-integrable functions w.r.t.\ \( \pi_0 \).
Assume that the function
\( g : \mathcal{Y} \times \mX \to (0, \infty) \)
satisfies the following two conditions:
\begin{enumerate}[label=\roman*)]
    \item \( g \) is jointly measurable w.r.t.\ the product \(\sigma\)-algebra
    \( \mathcal{B}(\mathcal{Y}) \otimes \mathcal{B}(\mX) \),
    \item for every \( y \in \mathcal{Y} \), the section
    \( x \mapsto g(y \mid x) \) belongs to \( L^2(\pi_0) \).
\end{enumerate}
Then the map \( \ell^{\, d_x} : \mathcal{Y} \to L^2(\pi_0) \) defined as
\[
\ell^{\, d_x}_y := \frac{g_y^{d_x}( \cdot)}{\int_{\mX} g_y^{d_x}(x) \, \pi_0 (\sd x)}
\]
is strongly \( \eta \)-measurable.
\end{lemma}

\begin{proof}
Since \( L^2(\pi_0) \) is a separable Hilbert space, by Pettis' Theorem
(cf.\ Theorem~1.1 in \cite{pettis1938integration}),
it suffices to show that \( \ell^{\, d_x} \) is weakly measurable.
This is equivalent to proving that the map
\[
y \mapsto \langle  \varphi, \ell_y^{\, d_x} \rangle_{L^2(\pi_0)}
\quad \text{is measurable for all }  \varphi \in L^2(\pi_0).
\]
Computing the inner product,
\[
\langle  \varphi, \ell_y^{\, d_x} \rangle_{L^2(\pi_0)}
=
\frac{1}{\pi_0(g_y^{d_x})}
\int_{\mX}
 \varphi(x)\, g_y^{d_x}(x) \, \pi_0(\sd x).
\]
Since the map \( (y,x) \mapsto  \varphi(x) g_y^{d_x}(x) \)
is jointly measurable, the Fubini--Tonelli Theorem
(cf.\ \cite{rudin1987real}, Theorem~8.8)
implies that
\[
y \mapsto \int_{\mX}  \varphi(x)\, g_y^{d_x}(x) \, \pi_0(\sd x)
\]
is measurable.
Moreover, since \( y \mapsto \pi_0(g_y^{d_x}) \) is measurable and strictly positive,
we conclude that
\[
y \mapsto \langle  \varphi, \ell_y^{\, d_x} \rangle_{L^2(\pi_0)}
\]
is measurable for all \(  \varphi \in L^2(\pi_0) \).
Therefore, \( \ell^{\, d_x} \) is weakly measurable.
\end{proof}

\subsection[Bounded functionals]{%
Boundedness of expectation functionals}\label{apx:Bochner-F}

Let $\mathcal{H}$ be a linear normed space over the field of real numbers $\mbR$. A mapping $\mathrm{F}: \mathcal{H} \mapsto\mbR$ is called a functional if it is linear; that is, for all $\upsilon, \varsigma \in \mathcal{H}$ and $\alpha, \beta \in \mbR$,
$$\mathrm{F}(\alpha\upsilon + \beta\varsigma) = \alpha \mathrm{F}(\upsilon) + \beta \mathrm{F}(\varsigma).$$
A functional $\mathrm{F}$ is called bounded if there exists a positive finite constant $C$ such that 
$$|\mathrm{F}(\upsilon)| \leq C \|\upsilon\|, \quad \text{for all } \,\upsilon \in \mathcal{H}.$$
 The norm of the functional $\mathrm{F}$ is defined as the smallest such constant $C$, given by
$$\|\mathrm{F}\| = \sup_{\upsilon \in \mathcal{H}, \upsilon \neq 0} \frac{|\mathrm{F}(\upsilon)|}{\|\upsilon\|} = \sup_{\|\upsilon\|=1} |\mathrm{F}(\upsilon)|.$$
A mapping $\mathrm{B}: \mathcal{H} \times \mathcal{H} \mapsto \mbR$ is called a bilinear form (or bifunctional) if it is linear in each argument separately.
A bilinear form $\mathrm{B}$ is called symmetric if $\mathrm{B}(\upsilon, \varsigma) = \mathrm{B}(\varsigma, \upsilon)$ for all $\upsilon, \varsigma \in \mathcal{H}$.
A bilinear form $\mathrm{B}$ is called bounded if there exists a positive finite constant $C$ such that 
$$|\mathrm{B}(\upsilon, \varsigma)| \leq C \|\upsilon\| \|\varsigma\|, \quad \text{for all } \, \upsilon, \varsigma \in \mathcal{H}.$$
 The norm of the bilinear form $\mathrm{B}$ is defined as
$$\|\mathrm{B}\| = \sup_{\upsilon, \varsigma \in \mathcal{H}, \upsilon, \varsigma \neq 0} \frac{|\mathrm{B}(\upsilon, \varsigma)|}{\|\upsilon\| \|\varsigma\|}= \sup_{\|\upsilon\|=1, \|\varsigma\|=1} |\mathrm{B}(\upsilon, \varsigma)|.$$
A quadratic form $\mathrm{Q}: \mathcal{H} \mapsto \mbR$ is induced by a symmetric bilinear form $\mathrm{B}$ by
$$\mathrm{Q}(\upsilon) = \mathrm{B}(\upsilon, \upsilon), \quad \text{for all } \, \upsilon \in \mathcal{H}.$$
The quadratic form $\mathrm{Q}$ is called bounded if there exists a positive finite constant $C$ such that 
$$|\mathrm{Q}(\upsilon)| \leq C \|\upsilon\|^2, \quad \text{for all } \, \upsilon \in \mathcal{H}.$$
The norm of the quadratic form $\mathrm{Q}$ is defined as
$$\|\mathrm{Q}\| = \sup_{\upsilon \in \mathcal{H}, \upsilon \neq 0} \frac{|\mathrm{Q}(\upsilon)|}{\|\upsilon\|^2}= \sup_{\|\upsilon\|=1} |\mathrm{Q}(\upsilon)|.$$

\begin{remark}\label{rem:constant_ext}
The space $L^2(\pi_0)$ embeds canonically into
$L^\infty(\mathcal{Y};L^2(\pi_0))$ by identifying each
$\varphi\in L^2(\pi_0)$ with the $y$-independent function
$\widetilde{\varphi}:\mathcal{Y}\times\mX\to\mathbb{R}$ defined by
\[
\widetilde{\varphi}(y,x)=\varphi(x).
\]
This embedding is isometric, since
\[
\|\widetilde{\varphi}\|_{L^\infty(\mathcal{Y};L^2(\pi_0))}
=
\|\varphi\|_{L^2(\pi_0)}.
\]
Hence
\[
L^2(\pi_0)\subset L^\infty(\mathcal{Y};L^2(\pi_0)),
\]
and $L^2(\pi_0)$ can be identified with the subspace of
$L^\infty(\mathcal{Y};L^2(\pi_0))$ consisting of functions that
are constant in the variable $y$.
\end{remark}

\begin{lemma}\label{lem:bounded_functional}
Define the functional $\mathrm{F}$ and the bilinear form $\mathrm{B}$ on
$L^\infty\!\bigl(\mathcal{Y}; L^2(\pi_0)\bigr)$ by
\[
\mathrm{F}(\upsilon)
:=
\int_{\mathcal{Y}}
\big\langle \upsilon , \ell^{\, d_x}_y \big\rangle_{L^2(\pi_0)} \, \eta(\sd y),
\qquad
\mathrm{B}(\upsilon,\varsigma)
:=
\int_{\mathcal{Y}}
\big\langle \upsilon , \ell^{\, d_x}_y \big\rangle_{L^2(\pi_0)}
\big\langle \varsigma , \ell^{\, d_x}_y \big\rangle_{L^2(\pi_0)}
\, \eta(\sd y),
\]
for all $\upsilon,\varsigma \in L^\infty\!\bigl(\mathcal{Y}; L^2(\pi_0)\bigr)$.
Then
\begin{enumerate}[label=\roman*)]
\item
$\mathrm{F}$ is bounded, if and only if, $\ell^{\, d_x} \in L^1(\mathcal{Y}; L^2(\pi_0))$;
\item
$\mathrm{B}$ is bounded, if and only if, $\ell^{\, d_x} \in L^2(\mathcal{Y}; L^2(\pi_0))$.
\end{enumerate}
In both cases, the optimal bound is
\[
K_p^{d_x}
:=
\int_{\mathcal{Y}}
\|\ell^{\, d_x}_y\|_{L^2(\pi_0)}^p \, \eta(\sd y),
\quad p\in \{1,2\}.
\]
\end{lemma}

\begin{proof}
We prove (ii); the proof of (i) is analogous.

Assume first that
$\ell^{\, d_x} \in L^2(\mathcal{Y}; L^2(\pi_0))$.
Then, by Cauchy--Schwarz inequality,
\begin{align}\label{ineq:K_optimal1}
\abs{B(\upsilon,\varsigma)} &\leq \|\upsilon\|_{L^\infty(\mY;L^2(\pi_0))}\|\varsigma\|_{L^\infty(\mY;L^2(\pi_0))} \int_{\mathcal{Y}} \|\ell_y^{\, d_x}\|_{L^2(\pi_0)}^2 \, \eta(\sd y) \notag\\ & =  K_2^{d_x} \|\upsilon\|_{L^\infty(\mY;L^2(\pi_0))}\|\varsigma\|_{L^\infty(\mY;L^2(\pi_0))}<\infty. \notag
\end{align}
In particular,  for \(\|\upsilon\|_{L^\infty(\mY;L^2(\pi_0))}= \|\varsigma\|_{L^\infty(\mY;L^2(\pi_0))}=1,\) we have  \beq\label{ineq:K_optimal1}
\norm{\mathrm{B}}= \sup_{\|\upsilon\|= \|\varsigma\|=1} |\mathrm{B}(\upsilon, \varsigma)| \leq K_2^{d_x}.
\eeq
Conversely, assume $\mathrm{B}$ is bounded.
Define
\[
\upsilon_y := \frac{g_y^{d_x}(\cdot)}{\|g_y^{d_x}\|_{L^2(\pi_0)}}.
\]
Then $\upsilon \in L^\infty(\mathcal{Y}; L^2(\pi_0))$ and $\|\upsilon_y\|_{L^\infty(\mY;L^2(\pi_0))} = 1$. Moreover, for each \(y \in \mathcal{Y}\)
\[
\langle \upsilon_y, \ell^{\, d_x}_y \rangle_{L^2(\pi_0)}
= \|\ell^{\, d_x}_y\|_{L^2(\pi_0)}.
\]
Indeed, a direct computation yields
\[
\langle \upsilon_y, \ell_y^{\,d_x} \rangle_{L^2(\pi_0)} = \frac{\int_{\mX} [g_y^{d_x}(x) ]^2 \pi_0(\sd x)}{\|g_y^{d_x}\|_{L^2(\pi_0)}  \pi_0(g_y^{d_x})}  
= \frac{\|g_y^{d_x}\|_{L^2(\pi_0)}}{\pi_0(l_y^{d_x})} = \|\ell_y^{\,d_x}\|_{L^2(\pi_0)}.
\]
Hence
\beq\label{ineq:K_optimal2}
\mathrm{Q}(\upsilon) =
\int_{\mathcal{Y}} \|\ell^{\, d_x}_y\|_{L^2(\pi_0)}^2 \, \eta(\sd y)
= K_2^{d_x}
\le \norm{\mathrm{B}},
\eeq
which implies $\ell^{\, d_x} \in L^2(\mathcal{Y}; L^2(\pi_0))$. From \eqref{ineq:K_optimal1} and \eqref{ineq:K_optimal2} \[K_2^{d_x}=\norm{\mathrm{B}}.\] 
\end{proof}

\begin{remark}\label{rem:proof_cor}
For all $\varphi \in L^2(\pi_0)$,
\[
\mathbb{E}[ \pi_Y(\varphi) ] = F(\widetilde \varphi),
\qquad
\mathbb{E}[ \pi_Y^2(\varphi) ] = \mathrm{Q}(\widetilde \varphi),
\]
where $\widetilde \varphi$ is the constant extension of $\varphi$ defined in Remark \ref{rem:constant_ext}.
\end{remark}


\section{Proof of Theorem~\ref{thm:thISR}}\label{appThIS3}

We first prove $(ii)\Rightarrow(i)$.

Suppose that $\ell^{\, d_x} \in L^2\big(\mathcal{Y};L^2(\pi_0)\big)$.
By Bayes' formula,
\[
\pi_Y(f)-\pi_Y^{N}(f)
=
\frac{\pi_0(f\,l_Y^{d_x})}{\pi_0(l_Y^{d_x})}
-
\frac{\pi_0^{N}(f\,l_Y^{d_x})}{\pi_0^{N}(l_Y^{d_x})}.
\]
Applying Proposition~\ref{MIP}, we obtain
\begin{equation*}
\left|\pi_Y(f)-\pi_Y^{N}(f)\right|
\le
\frac{1}{\pi_0(l_Y^{d_x})}
\left[
\left|\pi_0(f\,l_Y^{d_x})-\pi_0^{N}(f\,l_Y^{d_x})\right|
+
\|f\|_\infty
\left|\pi_0^{N}(l_Y^{d_x})-\pi_0(l_Y^{d_x})\right|
\right].
\end{equation*}
Since the second term corresponds to the special case $f\equiv 1$,
it suffices to bound the first term. Raising it to the power $p$, and using
the independence of $Y$ from the Monte Carlo samples $\{x^i\}_{i=1}^N$, we obtain
\begin{align*}
\mathbb{E}\!\left[
\frac{1}{\pi_0(l_Y^{d_x})^p}
\left|\pi_0(f\,l_Y^{d_x})-\pi_0^{N}(f\,l_Y^{d_x})\right|^p
\right]
&=
\int_{\mathcal{Y}}
\frac{1}{\pi_0(l_y^{d_x})^p}
\mathbb{E}\!\left[
\left|\pi_0(f\,l_y^{d_x})-\pi_0^{N}(f\,l_y^{d_x})\right|^p
\right]
\,\eta(\mathrm{d}y).
\end{align*}
Applying Lemma~\ref{lem:L2bound} together with Eq. \eqref{norm_link_f} gives
\[
\int_{\mathcal{Y}}
\frac{1}{\mathtt{c}^p\, \pi_0(g_y^{d_x})^p}
\mathbb{E}\!\left[
\left|\pi_0(f\,l_y^{d_x})-\pi_0^{N}(f\,l_y^{d_x})\right|^p
\right]
\,\eta(\mathrm{d}y)
\le
\frac{\widetilde B_p\|f\|_\infty^p}{N^{p/2}}
\int_{\mathcal{Y}}
\|\ell^{\,d_x}_y\|_{L^2(\pi_0)}^p
\,\eta(\mathrm{d}y).
\]
For $f\equiv 1$, the same arguments gives
\[
\int_{\mathcal{Y}}
\frac{1}{\mathtt{c}^p \,\pi_0(g_y^{d_x})^p}
\mathbb{E}\!\left[
\left|\pi_0(l_y^{d_x})-\pi_0^{N}(l_y^{d_x})\right|^p
\right]
\,\eta(\mathrm{d}y)
\le
\frac{\widehat B_p}{N^{p/2}}
\int_{\mathcal{Y}}
\|\ell^{\,d_x}_y\|_{L^2(\pi_0)}^p
\,\eta(\mathrm{d}y).
\]
Combining the two bounds after using the elementary inequality
$\lvert x+y\rvert^p \le 2^{p-1}(\lvert x\rvert^p+\lvert y\rvert^p)$,
we conclude that
\[
\mathbb{E}\!\left[
\left|\pi_Y(f)-\pi_Y^{N}(f)\right|^p
\right]
\le
\frac{B_p\|f\|_\infty^p}{N^{p/2}}
\int_{\mathcal{Y}}
\|\ell^{\,d_x}_y\|_{L^2(\pi_0)}^p
\,\eta(\mathrm{d}y).
\]
We now prove $(i)\Rightarrow(ii)$.

Assume that inequality~\eqref{eq:Iff_error} holds for all $f\in B(\mX)$.
Observe first that
\begin{align*}
(\pi_Y(f)-\pi_Y^{N}(f))^2
=
\frac{1}{\pi_0(l_Y^{d_x})^2}
\Big[
&\big(\pi_0(f\,l_Y^{d_x})-\pi_0^N(f\,l_Y^{d_x})\big)^2 \notag\\
&+2\,\pi_Y^{N}(f)
\big(\pi_0(f\,l_Y^{d_x})-\pi_0^N(f\,l_Y^{d_x})\big)
\big(\pi_0^N(l_Y^{d_x})-\pi_0(l_Y^{d_x})\big) \notag\\
&+\big(\pi_Y^{N}(f)\big)^2
\big(\pi_0^N(l_Y^{d_x})-\pi_0(l_Y^{d_x})\big)^2
\Big].
\end{align*}
Taking expectation and using the independence between $Y$ and the Monte
Carlo samples, a necessary condition for
\[
\mathbb{E}\big[(\pi_Y(f)-\pi_Y^{N}(f))^2\big]<\infty
\]
is that
\begin{equation}\label{eq:71}
\int_{\mathcal{Y}}
\frac{1}{\pi_0(l_y^{d_x})^2}
\mathbb{E}\!\left[
\big(\pi_Y^{N}(f)\big)^2
\big(\pi_0^N(l_y^{d_x})-\pi_0(l_y^{d_x})\big)^2
\right]
\eta(\mathrm{d}y)
<\infty .
\end{equation}
Choose $f$ positive and bounded away from zero, i.e.\ there exists a constant $\epsilon>0$
such that  $f(x)\ge \epsilon$ for all $x\in\mX$. Then Eq. \eqref{eq:71} yields

\begin{equation}\label{eq:MbF}
\int_{\mathcal{Y}}
\frac{1}{\pi_0(l_y^{d_x})^2}
\mathbb{E}\!\left[
\big(\pi_0^N(l_y^{d_x})-\pi_0(l_y^{d_x})\big)^2
\right]
\eta(\mathrm{d}y)
<\infty .
\end{equation}
For fixed $y$, define the iid\ zero-mean r.v.s
\[
U_i:=\frac{1}{N}\big(\pi_0(l_y^{d_x})-l_y^{d_x}(x^i)\big),
\qquad i=1,\dots,N,
\]
so that
\[
\sum_{i=1}^N U_i=\pi_0(l_y^{d_x})-\pi_0^N(l_y^{d_x}).
\]
Applying the lower bound in the Marcinkiewicz--Zygmund inequality yields
\[
A_2\,\mathbb{E}\!\left[\sum_{i=1}^N U_i^2\right]
\le
\mathbb{E}\!\left[\left|\sum_{i=1}^N U_i\right|^2\right],
\]
where $A_2>0$ only depends on $p=2$. Since
\[
\mathbb{E}\!\left[\sum_{i=1}^N U_i^2\right]
=
\frac{1}{N}
\big(\pi_0([l_y^{d_x}]^2)-\pi_0(l_y^{d_x})^2\big),
\]
we obtain
\[
\frac{A_2}{N}
\big(\pi_0([l_y^{d_x}]^2)-\pi_0(l_y^{d_x})^2\big)
\le
\mathbb{E}\!\left[
\big|\pi_0(l_y^{d_x})-\pi_0^N(l_y^{d_x})\big|^2
\right].
\]
Dividing by \(\pi_0(\ell_y^{\,d_x})^2\), integrating over \(\mathcal{Y}\), and applying \eqref{eq:MbF} along with Remark~\ref{rem:link_norm}, we obtain, in particular,
\[
\int_{\mathcal{Y}}
\|\ell_y^{\,d_x}\|_{L^2(\pi_0)}^2 \,\eta(\mathrm{d}y)
=
\int_{\mathcal{Y}}
\frac{\pi_0([l_y^{d_x}]^2)}{\pi_0(l_y^{d_x})^2}\,\eta(\mathrm{d}y)
<\infty,
\]
which proves that $\ell^{\,d_x}\in L^2(\mathcal{Y};L^2(\pi_0))$.
\qed

\section{Proof of Theorem \ref{thIS2R}}\label{ap:T4.3}
The result follows directly from the sufficiency argument in the proof of  Theorem \ref{thm:thISR} (i.e., (ii) $\Rightarrow$ (i)), which remains valid for both \(p = 1\) and \(p = 2\). Under the assumption that there is a polynomial \(P_{n,p}(d_x)\) bounding \(K_p^{d_x}\), the conclusion is straightforward.
\qed

\section{Numerical example}\label{ap:Numerical_Example}

To illustrate the dimensional behavior predicted by the theoretical bounds developed throughout the paper, we consider a family of linear Gaussian models for which the relevant spectral quantities can be controlled explicitly as the state dimension \(d_x\) varies.

We present two numerical examples, both with fixed observation dimension \(d_y\), but with different spectral scalings. In the first example, the model is normalized so that the quantity
\[
\sigma_1^2(A^{d_x})\lambda_1(\Sigma_x^{d_x})
\]
grows polynomially with respect to \(d_x\). As predicted by the theory, this produces polynomial growth of the Bochner constant \(K_2^{d_x}\) and of the corresponding importance sampling error.

In the second example, the normalization is chosen so that the previous spectral quantity remains uniformly bounded with respect to \(d_x\). In this regime, the theoretical bounds predict that both the Bochner constant and the importance sampling error remain uniformly stable as the state dimension increases.

Together, these examples illustrate how the spectral structure of the model determines the dimensional scaling of the average errors with respect to the observation distribution.

Throughout both examples, we fix
\[
d_y=3,
\qquad
R=I_{d_y},
\]
and, for each
\[
d_x \in \{5,10,20,50,10^2,5\cdot10^2,10^3,10^4\},
\]
we generate random matrices
we generate random matrices
\[
\widetilde A^{d_x} \in \mbR^{d_y\times d_x},
\qquad
\widetilde B^{d_x} \in \mbR^{d_x\times d_x},
\]
whose entries are independent and uniformly distributed in $(0,1)$. Since $\sigma_1(\widetilde A^{d_x})=\|\widetilde A^{d_x}\|_2 $, this assumption is enough to ensure that $\sigma_1(\widetilde A^{d_x}) > 0$. We then define
\[
\widetilde \Sigma_x^{d_x}
: =
\widetilde B^{d_x}(\widetilde B^{d_x})^\top
+
\eta I_{d_x},
~~\text{with}~~
\eta=10^{-2}.
\]
Thus, $\widetilde \Sigma_x^{d_x}$ is symmetric positive definite. Indeed, for every \(x\neq 0\),
\[
x^\top \widetilde \Sigma_x^{d_x} x
=
x^\top \widetilde B^{d_x} (\widetilde B^{d_x})^\top x 
+
\eta x^\top x
=
\| (\widetilde B^{d_x})^\top x\|^2
+
\eta \|x\|^2.
\]
Since
\(
\eta \|x\|^2 > 0,
\)
we obtain
\[
x^\top \widetilde \Sigma_x^{d_x} x > 0,
\qquad
\forall x\neq 0.
\]
Hence,
\(
\widetilde \Sigma_x^{d_x}\succ 0.
\)
\subsection{Polynomial scaling}
The model matrices are obtained through the spectral normalizations
\beq\label{eq:Control_A_Sigmax}
A^{d_x}
=
d_x^{1/4}
\frac{\widetilde A^{d_x}}{\sigma_1(\widetilde A^{d_x})},
\qquad
\Sigma_x^{d_x}
=
d_x^{1/2}
\frac{\widetilde \Sigma_x^{d_x}}
{\lambda_1(\widetilde \Sigma_x^{d_x})}.
\eeq
By \eqref{eq:Control_A_Sigmax} and the homogeneity of singular values and eigenvalues, namely
\[
\sigma_1(cM)=c\,\sigma_1(M),
\qquad
\lambda_1(cM)=c\,\lambda_1(M),
\qquad c>0,
\]
we obtain the exact identities
\[
\sigma_1(A^{d_x})=d_x^{1/4},
\qquad
\lambda_1(\Sigma_x^{d_x})=d_x^{1/2}.
\]
Consequently,
\[
\sigma_1^2(A^{d_x})\lambda_1(\Sigma_x^{d_x})
=
d_x.
\]
This quantity is the key mechanism behind the polynomial behavior observed below. Indeed, applying the upper bound in \eqref{K_2Up_Bound} to
\[
K_2^{d_x}
\leq \left(
1+
\frac{
\lambda_{1}(\Sigma_x^{d_x})
}{
\lambda_{d_y}(R)
}
\sigma_1^2(A^{d_x})
\right)^{d_y}\]
yields
\beq\label{ineq:Sigma_y_pol}
K_2^{d_x}
\le 
(1+d_x)^{d_y}=:P(d_x).
\eeq
For each value of $d_x$, we choose a prior mean
\[
\mu_x^{d_x}\in \mbR^{d_x},
\]
which is generated for the simulation as a sample from a standard $d_x$-dimensional Gaussian distribution. We then sample a  state
\[
x^\star \sim \mathcal{N}(\mu_x^{d_x},\Sigma_x^{d_x}),
\]
and generate one observation
\[
y=A^{d_x}x^\star+v,
\qquad
v\sim \mathcal{N}(0,I_{d_y}).
\]

Since the model is linear and Gaussian, the posterior distribution of $X$ is available in closed form. In particular, the posterior mean is
\[
m_y^{d_x}
=
\mu_x^{d_x}
+
\Sigma_x^{d_x}[A^{d_x}]^\top
\left(
A^{d_x}\Sigma_x^{d_x}[A^{d_x}]^\top+I_{d_y}
\right)^{-1}
\left(
y-A^{d_x}\mu_x^{d_x}
\right).
\]
We compare this exact posterior mean with an importance sampling approximation with $N=400$ particles sampled from the prior,
\[
x^1,\ldots,x^N
\sim
\mathcal{N}(\mu_x^{d_x},\Sigma_x^{d_x}).
\]
For the realized observation $y$, the normalized weights are
\[
w_i(y)
=
\frac{
\exp\left\{
-\frac12 \|y-A^{d_x}x^i\|^2
\right\}
}{
\sum_{j=1}^N
\exp\left\{
-\frac12 \|y-A^{d_x}x^j\|^2
\right\}
},
\qquad i=1,\ldots,N,
\]
and the corresponding importance sampling approximation of the posterior mean is
\[
m_{y,N}^{d_x}
=
\sum_{i=1}^N w_i(y)x^i.
\]
For each dimension $d_x$, we compute the Euclidean error
\[
\mE_N^{d_x} := \|m_{y,N}^{d_x}-m_y^{d_x}\|_2,
\]
as well as the Bochner constant, which by Eq.~\eqref{eq:Boch_Det} takes the form
\[
K_2^{d_x}
=
\frac{
\abs{
A^{d_x}\Sigma_x^{d_x}[A^{d_x}]^\top+I_{d_y}
}
}{
\abs{I_{d_y}}
}
=
\abs{
A^{d_x}\Sigma_x^{d_x}[A^{d_x}]^\top+I_{d_y}
}.
\]
Recall that, by   \eqref{ineq:Sigma_y_pol}
\[
K_2^{d_x}
\le
(1+d_x)^{d_y},
\]
and therefore the theoretical upper bound grows almost polynomially. Recall that, in this example, the observation dimension is fixed and equal to \(d_y=3\).

Figure \ref{F1} compares $K_2^{d_x}$, the importance sampling error $\mE_N^{d_x}$, the reference polynomial $P(d_x)=(1+d_x)^{3}$, and the exponential reference $e^{d_x}$. The exponential curve is displayed only for $d_x\leq 50$, since it grows much faster than the polynomial quantities.

\begin{figure}[htb]
\centering
\includegraphics[width=0.75\textwidth]{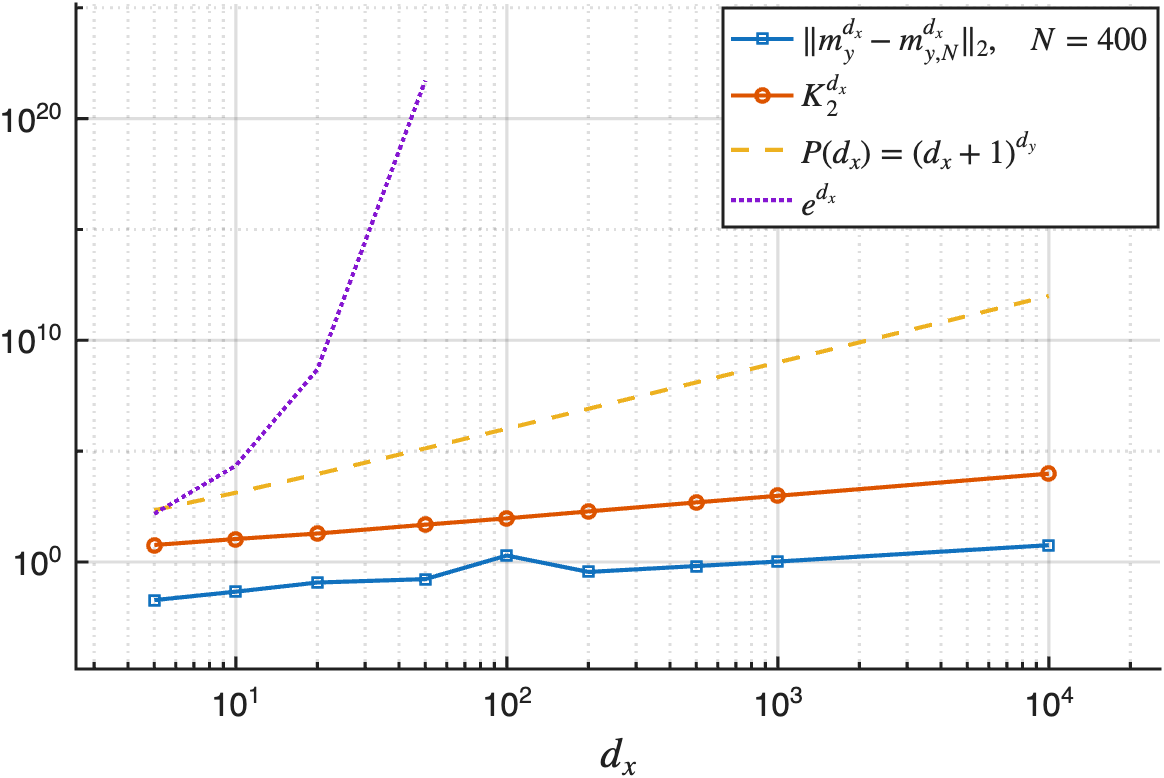}
\caption{The Bochner constant $K_2^{d_x}$ and the importance sampling error $\mE_N^{d_x}=\|m_{y,N}^{d_x}-m_y^{d_x}\|_2$ as a function of the state dimension $d_x$, for $d_y=3$ and $N=400$. The polynomial reference predicted $P(d_x)=(1+d_x)^{d_y}$ and the exponential reference $e^{d_x}$ are shown for comparison. The results illustrate the polynomial growth predicted by the theoretical bounds.}\label{F1}
\end{figure}

\subsection{Uniform bound}
If instead of allowing the spectral quantities to grow with $d_x$, we impose uniform bounds of the form
\[
\sigma_1(A^{d_x}) \leq C_A,
\qquad
\lambda_1(\Sigma_x^{d_x}) \leq C_\Sigma,
\qquad \text{for all } d_x,
\]
for some constants $C_A, C_\Sigma > 0$ independent of $d_x$, which can be achieved by a straightforward modification of the normalization in \eqref{eq:Control_A_Sigmax}, then the key quantity governing the growth satisfies
\[
\sigma_1^2(A^{d_x}) \lambda_1(\Sigma_x^{d_x}) \leq C_A^2 C_\Sigma.
\]
Consequently, the determinant bound in \eqref{K_2Up_Bound} yields
\[
K_2^{d_x}
\leq
\left(
C_A^2 C_\Sigma + 1
\right)^{d_y}:=\mK_2,
\] therefore, the Bochner constant remains uniformly bounded w.r.t. $d_x$ 
and the corresponding error of the importance sampling approximation, $\mE_N^{d_x}$, remains uniformly bounded as well. This behavior is illustrated in Figure~\ref{fig:LinearGaussUniform}, where both quantities remain stable as $d_x$ increases, in contrast with the polynomial growth observed in Figure \ref{F1}.

\begin{figure}[htb]
\centering
\includegraphics[width=0.75\textwidth]{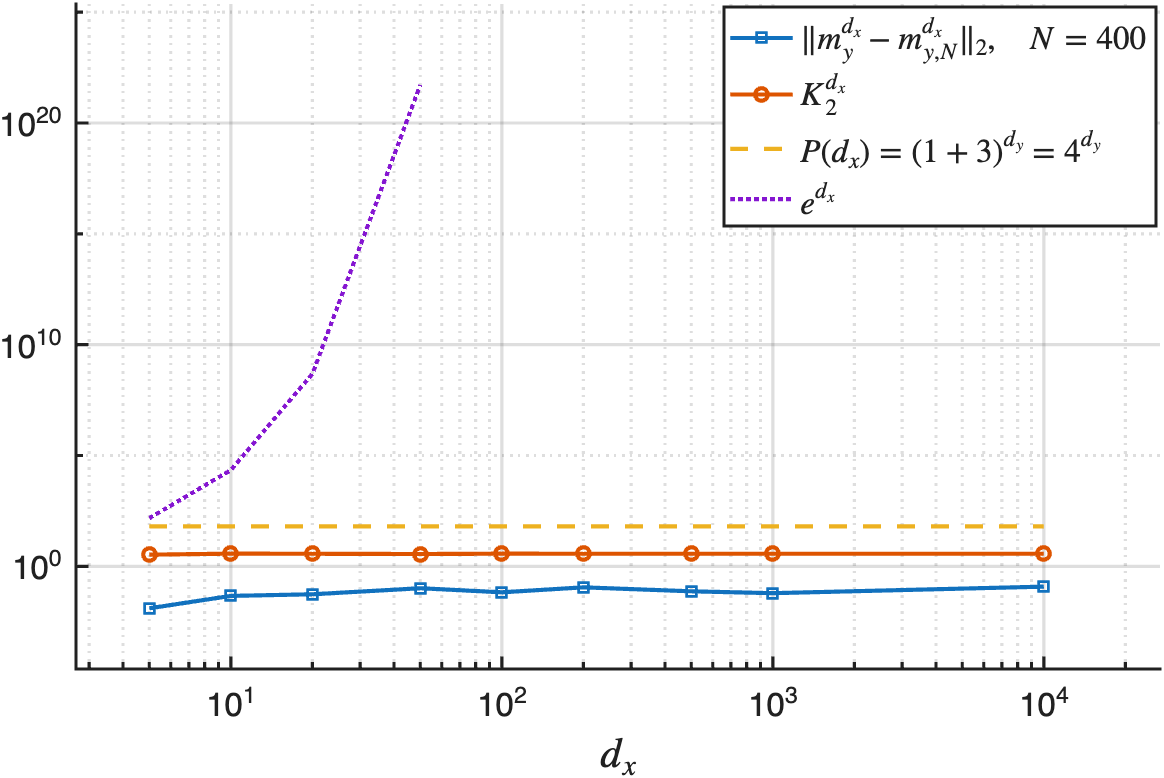}
\caption{The Bochner constant $K_2^{d_x}$ and the importance sampling error $\mE_N^{d_x}=\|m_{y,N}^{d_x}-m_y^{d_x}\|_2$ as a function of the state dimension $d_x$, for $d_y=3$ and $N=400$. The uniform theoretical bound $\mathcal K_2=(1+3)^{d_y}$ and the exponential reference $e^{d_x}$ are shown for comparison. The results illustrate the uniform behavior predicted by the theoretical bounds, where both the Bochner constant and the importance sampling error remain stable as the state dimension increases.}
\label{fig:LinearGaussUniform}
\end{figure}

\section{Elliptically symmetric likelihood families}\label{app:Colours_Noise}

\subsection{Mahalanobis-distance geometry}\label{sub:Mahalanobis-geo}

Throughout this section, geometric arguments are conducted w.r.t. the Mahalanobis norm $\|\cdot\|_R$. We assume the observation map $h^{d_x}: \mX\subseteq \mbR^{d_x} \to \mathcal{Y}$ is uniformly bounded in the Euclidean norm; specifically, there exists a $d_x$-dependent constant 
\[
M(d_x) := \sup_{x \in \mX} \|h^{d_x}(x)\| < \infty.
\]
Let $R^{1/2}$ denote the unique symmetric positive definite square root of $R$, such that  $\|y\|_R = \|R^{1/2} y\|$. If $\lambda_{1}(R)$ denotes the largest eigenvalue of $R$, the induced operator 2-norm of $R^{1/2}$ is $\|R^{1/2}\| = \sqrt{\lambda_{1}(R)}$. It follows from the submultiplicativity of the norm that
\[
\|h^{d_x}(x)\|_R = \|R^{1/2} h^{d_x}(x)\| \le \sqrt{\lambda_{1}(R)} \, \|h^{d_x}(x)\| \le \sqrt{\lambda_{1}(R)} M(d_x).
\]
Accordingly, we define the uniform bound in the $R$-norm as
\[
M_R^{d_x} := \sup_{x \in \mX} \|h^{d_x}(x)\|_R \le \sqrt{\lambda_{1}(R)} M(d_x).
\]
By the triangle inequality, for all $x \in \mX$ and $y \in \mathcal{Y}$, we have
\begin{equation}\label{eq:triangle-R}
\|y - h^{d_x}(x)\|_R 
\le \|y\|_R + M_R^{d_x}.
\end{equation}

\emph{Propagation through $\psi$ and $\phi$.}
Let $\psi: [0,\infty) \to \mathbb{R}$ be non-decreasing and $\phi: \mathbb{R} \to (0,\infty)$ be non-increasing. Applying $\psi$ to the inequality in \eqref{eq:triangle-R} yields a bound that is uniform in $x \in \mathcal{X}$
\begin{equation*}
\psi\big(\|y - h^{d_x}(x)\|_R\big) \le \psi\big(\|y\|_R + M_R^{d_x}\big).
\end{equation*}
Since $\phi$ is non-increasing, the inequality direction reverses, providing the lower bound
\begin{equation}\label{eq:final-general-bound-R}
\phi\big(\psi(\|y\|_R + M_R^{d_x})\big) 
\le \pi_0 \left( \phi\big(\psi(\|y - h^{d_x}(x)\|_R)\big) \right) = \pi_0 (g_y^{d_x}).
\end{equation}
The bound in \eqref{eq:final-general-bound-R} establishes that the normalization constant of any elliptically symmetric  likelihood of the form \eqref{eq:general-radial-likelihood-fixed} is bounded below by a radial profile depending solely on the Mahalanobis norm of $y$ and the uniform bound $M_R^{d_x}$.

\subsection{Proof of Theorem \ref{thm:Colour_Noises}}\label{sec:thm:Colour_Noises}

The squared likelihood for the elliptically symmetric  family is expressed as
\begin{equation*}
g(y \mid x)^2 = \phi^2 \big( \psi(\|y - h^{d_x}(x)\|_R) \big).
\end{equation*}
From the general bounds established in \eqref{eq:ratio} and \eqref{eq:final-general-bound-R}, the second moment of the link function satisfies
\begin{equation}
\label{eq:int_ratio_rewritten}
\mathbb{E} \left[ \| \ell_Y^{\, d_x} \|_{L^2(\pi_0)}^2 \right] \le \int_{\mX} \left(\int_{\mathcal{Y}} \frac{g_y^{d_x}(x)^2}{\phi \big( \psi(\|y\|_R + M_R^{d_x}) \big)} \, \mathrm{d}y \right) \, \pi_0(\mathrm{d}x).
\end{equation}
Applying the change of variables $\tilde{y} = y - h^{d_x}(x)$ to the inner integral, we obtain
\begin{equation*}
I(x) :=\int_{\mathcal{Y}} \frac{\phi^2 \big( \psi(\|\tilde{y}\|_R) \big)}{\phi \big( \psi(\|\tilde{y} + h^{d_x}(x)\|_R + M_R^{d_x}) \big)} \, \mathrm{d}\tilde{y}.
\end{equation*}
By the triangle inequality for the Mahalanobis norm and the definition of $M_R^{d_x}$ in \eqref{eq:MR-def}, we have
\[
\|\tilde{y} + h^{d_x}(x)\|_R + M_R^{d_x}  
\le \|\tilde{y}\|_R + 2M_R^{d_x}.
\]
Recall that $\psi$ is non-decreasing and $\phi$ is non-increasing. Consequently, the denominator in the integrand is lower bounded by
\[
\phi \big( \psi(\|\tilde{y} + h^{d_x}(x)\|_R + M_R^{d_x}) \big) \ge \phi \big( \psi(\|\tilde{y}\|_R + 2M_R^{d_x}) \big).
\]
This yields a uniform bound for the integrand that is independent of $x \in \mX$
\begin{equation*}
\frac{\phi^2 \big( \psi(\|\tilde{y}\|_R) \big)}{\phi \big( \psi(\|\tilde{y} + h^{d_x}(x)\|_R + M_R^{d_x}) \big)} \le \frac{\phi^2 \big( \psi(\|\tilde{y}\|_R) \big)}{\phi \big( \psi(\|\tilde{y}\|_R + 2M_R^{d_x}) \big)}.
\end{equation*}
Defining the map
\begin{equation}\label{eq:radial_k}
k(y) := \frac{\phi^2 \big( \psi(\|y\|_R) \big)}{\phi \big( \psi(\|y\|_R + 2M_R^{d_x}) \big)},
\end{equation}
and substituting the uniform bound \eqref{eq:radial_k} into \eqref{eq:int_ratio_rewritten}, the second moment of the link function simplifies to 
\begin{equation*}
\mathbb{E} \left[ \| \ell_Y^{\, d_x} \|_{L^2(\pi_0)}^2 \right] \le \int_{\mX} \left( \int_{\mathcal{Y}} k(y) \, \mathrm{d}y \right) \pi_0(\mathrm{d}x).
\end{equation*}
Since the inner integral is independent of $x$ and $\pi_0$ is a probability measure on $\mX$, the expression reduces to 
\begin{equation*}
\mathbb{E} \left[ \| \ell_Y^{\, d_x} \|_{L^2(\pi_0)}^2 \right] \le \int_{\mathcal{Y}} k(y) \, \mathrm{d}y.
\end{equation*}

To evaluate this, we transform to radial coordinates $r = \|y\|_R$. The Lebesgue measure on $\mathcal{Y}$ decomposes as $\mathrm{d}y = \mathcal{S}_R^{d_y} r^{d_y-1} \mathrm{d}r$, where $\mathcal{S}_R^{d_y} = |R|^{1/2} \frac{2\pi^{d_y/2}}{\Gamma(d_y/2)}$ denotes the weighted surface area of the $d_y$--unit sphere. It follows that
\begin{equation}
\label{eq:radial_condition_final}
\int_{\mathcal{Y}} k(y) \, \mathrm{d}y = \mathcal{S}_R^{d_y} \int_0^\infty \frac{\phi^2 \big( \psi(r) \big)}{\phi \big( \psi(r + 2M_R^{d_x}) \big)} \, r^{d_y-1} \, \mathrm{d}r.
\end{equation}
Thus, the $L^2$-integrability of the link function is guaranteed if the one-dimensional radial integral in \eqref{eq:radial_condition_final} is finite. This condition is determined solely by the tail behavior of the profile $\phi \circ \psi$, the observation dimension $d_y$, and the observation map bound $M_R^{d_x}$. \qed

\subsection{Proof of Corollary \ref{cor:Colour_Noises}}\label{sec:cor:Colour_Noises}

The integrability of the link function is governed by the radial integral defined in~\eqref{eq:radial_condition_final}. 
We analyze separately the two classes of radial profiles described in Tables~A and~B of Remark~\ref{rem:Colour_noises}.

\subsubsection{Case 1: Exponential radial profiles}
Let $\phi(s) = \mathrm{C}\, e^{-s}$ and $\psi(r) = a r^\beta$ for constants $\mathrm{C}, a, \beta > 0$. The radial integral in \eqref{eq:radial_condition_final} then takes the form
\begin{equation*}
\mathcal{J} := \mathrm{C} \,\mathcal{S}_R^{d_y} \int_0^\infty \exp \left( \psi(r + 2M_R^{d_x}) - 2\psi(r) \right) r^{d_y-1} \, \mathrm{d}r.
\end{equation*}
The integrability is determined by the asymptotic behavior of the exponent \[Q(r) :=a[(r + 2M_R^{d_x})^\beta - 2r^\beta].\] Factoring out the dominant term, we obtain
\[
Q(r) = a r^\beta \left[ \left( 1 + \frac{2M_R^{d_x}}{r} \right)^\beta - 2 \right].
\]
Since $\lim_{r \to \infty} (1 + 2M_R^{d_x}/r)^\beta = 1$, for any $\varepsilon \in (0, 1)$ there exists a radius $R_\varepsilon$ such that  for all $r \ge R_\varepsilon$, $Q(r) \le -a(1 - \varepsilon)r^\beta$. Setting $\varepsilon = 1/2$ yields $Q(r) \le -\frac{a}{2}r^\beta$ for $r \ge 4M_R^{d_x}$. We decompose the integral $\mathcal{J}$ into a finite part $\mathcal{J}_1$ and a tail part $\mathcal{J}_2$:
\begin{enumerate}
    \item \textit{Finite part:} For $r \in [0, 4M_R^{d_x}]$, the exponent is bounded by $Q(r) \le a(6M_R^{d_x})^\beta$. Thus,
\begin{equation*}
    \mathcal{J}_1 \le \mathrm{C}\,  \mathcal{S}_R^{d_y} \exp\left( a(6M_R^{d_x})^\beta \right) \frac{(4M_R^{d_x})^{d_y}}{d_y}.
    \end{equation*}
    \item \textit{Tail part:} For $r > 4M_R^{d_x}$, the integrand is bounded by $\exp(-br^\beta)r^{d_y-1}$ with $b = a/2$. The change of variables $t = br^\beta$ yields
\begin{equation*}
    \mathcal{J}_2 \le \frac{\mathrm{C}\, \mathcal{S}_R^{d_y}}{\beta b^{d_y/\beta}} \Gamma\left( \frac{d_y}{\beta}, b(4M_R^{d_x})^\beta \right),
    \end{equation*}
    where $\Gamma(s, x)$ is the upper incomplete Gamma function.
\end{enumerate}
The sum $\mathcal{J}_1 + \mathcal{J}_2$ provides a finite bound $K_2^{d_x}$. If $\sup_{d_x\in \mbN} M_R^{d_x} \le M$, the bound is independent of $d_x$.
\subsubsection{Case 2: Polynomial-tail profiles}
Let $\phi(s) = \mathrm{C}\,  s^{-\alpha}$ and $\psi(r) = 1 + ar^2$ for $\alpha, a > 0$. The integral \eqref{eq:radial_condition_final} becomes
\begin{equation}
\label{eq:I_poly_ratio}
\mathcal{J} = \mathrm{C}\, \mathcal{S}_R^{d_y} \int_0^\infty \left( \frac{1 + a(r + 2M_R^{d_x})^2}{(1 + ar^2)^2} \right)^\alpha r^{d_y-1} \, \mathrm{d}r.
\end{equation}
Using the inequality $(r+k)^2 \le 2(r^2+k^2)$, the numerator satisfies
\[
1 + a(r + 2M_R^{d_x})^2 \le 2(1 + a(2M_R^{d_x})^2)(1 + ar^2).
\]
Substituting this into \eqref{eq:I_poly_ratio}, we obtain
\[
\mathcal{J} \le 2^\alpha (1 + a(2M_R^{d_x})^2)^\alpha \, \mathrm{C}\,  \mathcal{S}_R^{d_y} \int_0^\infty (1 + ar^2)^{-\alpha} r^{d_y-1} \, \mathrm{d}r.
\]
 This integral converges for $\alpha > d_y/2$ and evaluates to a constant. Indeed, 
by the normalization of the likelihood, $\mathrm{C}\,  \mathcal{S}_R^{d_y} \int_0^\infty (1 + ar^2)^{-\alpha} r^{d_y-1} \, \mathrm{d}r = 1$. Thus, the second moment is bounded by
\begin{equation*}
\mathbb{E}\left[ \| \ell_Y^{\, d_x} \|_{L^2(\pi_0)}^2 \right] = K_2^{d_x} \le 2^\alpha (1 + a(2M_R^{d_x})^2)^\alpha,
\end{equation*}
completing the proof.
Finally, note that if $M_R^{d_x} \le P(d_x)$ for some polynomial $P$ in $d_x$, then $K_2^{d_x}$ satisfies~\eqref{eq:K_p_leq_P}. \hfill \qed

\end{appendix}


\begin{acks}[Acknowledgments]

\end{acks}

\begin{funding}
FG and JM acknowledge the support of Spain’s {\em Agencia Estatal de Investigación} (ref. PID2024-158181NB-I00 NISA), funded by MCIN/AEI/10.13039/501100011033 and by the ERDF (“A way of making Europe”); {\em Comunidad de Madrid} (IDEA-CM project, ref. TEC-2024/COM-89), and the Office of Naval Research (award N00014-22-1-2647). The work of VE is supported by ARL/ARO under grant W911NF-22-1-0235 and by the Advanced Research + Invention Agency (ARIA).
\end{funding}

\bibliographystyle{imsart-number} 

\bibliography{bibliografia}       

\end{document}